\documentclass[11pt]{article}

\usepackage[round]{natbib} % Bibliography / reference package
\usepackage{anysize}
\marginsize{2cm}{2.5cm}{2.5cm}{2cm}

\usepackage{lmodern} % use Latin Modern for that LaTeX look

\usepackage{chngpage}
\usepackage{graphicx}
\usepackage{url}
\usepackage{amsmath}
\usepackage{amsfonts}
\usepackage{amssymb}

\usepackage{algorithm2e}

\usepackage{here}
\usepackage{amsfonts}
\usepackage{amsthm}
\usepackage{comment}
\usepackage{enumitem}
\usepackage[blocks]{authblk}
\usepackage{mathrsfs}
\usepackage{bm}
\usepackage{hyperref}
\hypersetup{
	setpagesize=false,
	bookmarksnumbered=true,%
	bookmarksopen=true,%
	colorlinks=false,%
	linkcolor=blue,
	citecolor=red,
}

\usepackage[textwidth=3cm, textsize=footnotesize]{todonotes} %%%%new command for comments

\newcommand{\R}{\mathbb{R}}
\newcommand{\N}{\mathbb{N}}
\newcommand{\Z}{\mathbb{Z}}
\newcommand{\E}{\mathbb{E}}

\newcommand{\Filt}{\mathcal{F}}
\newcommand{\Tdm}{\mathscr{T}}

\newcommand{\plim}{\xrightarrow{p}}
\newcommand{\dlim}{\xrightarrow{d}}

\def\intp[#1]{\lfloor #1 \rfloor}
\newcommand{\eor}{%
	\relax\ifmmode
	\eqno{%
		\triangleleft}
	\else
	\begingroup
	\hfill $\triangleleft$
	\endgroup
	\fi}

\newcommand{\blue}[1]{#1} %%% plain text
\usepackage{xcolor}
\newcommand{\erase}[1]{} %%% erase from pdf

\definecolor{purple}{cmyk}{0.1,0.9,0,0.4}

\theoremstyle{definition}
\newtheorem{theorem}{Theorem}
\newtheorem{lemma}{Lemma}
\newtheorem{proposition}{Proposition}
\newtheorem{assumption}{Assumption}
\newtheorem{remark}{Remark}
\newtheorem{example}{Example}

\author[1]{Fumiya Akashi}
\author[2]{Konstantinos Fokianos}
\author[3]{Junichi Hirukawa}

\affil[1]{University of Tokyo}
\affil[2]{University of Cyprus}
\affil[3]{Nanzan  University}

{
	\makeatletter
	\renewcommand\AB@affilsepx{: \protect\Affilfont}
	\makeatother
	
	\affil[ ]{Email}
	
	\makeatletter
	\renewcommand\AB@affilsepx{, \protect\Affilfont}
	\makeatother
	
	\affil[1]{akashi@e.u-tokyo.ac.jp}
	\affil[2]{fokianos@ucy.ac.cy}
	\affil[3]{hirukawa@nanzan-u.ac.jp}
}

\title{Inference for Non-Stationary Heavy Tailed Time Series}
\begin{document}
	\maketitle

	\begin{abstract}
		\noindent
		We consider the problem of inference for linear autoregressive  heavy tailed  non-stationary
		time series. Using the framework of  time-varying infinite order moving average processes  we show that there   exists a suitable local approximation to a non-stationary heavy tailed process 
		by a stationary process with heavy tails. This enables us to introduce a local approximation-based estimator which estimates consistently time-varying parameters of the model at hand.  In addition,
		we  study a  robust self-weighting approach that enables proving  the asymptotic normality of the estimator without necessarily assuming finite variance for the underlying process. A multiplier bootstrap methodology is suggested for obtaining a consistent estimator of the variance for this new estimator. 
		Empirical evidence favoring this approach is provided.

		\bigskip
		
		\noindent 
 		\textbf{AMS Subject Classification}: Primary 62G08 \\
%		\noindent\textbf{JEL classification: C13, C22}\\
		\noindent 
		\textbf{Keywords}: heavy tails, infinite variance processes,  least absolute deviation estimator, local regression, non-stationarity.   
		
		\pagebreak
	\end{abstract}

\section{Introduction}
\label{Sec:Intro}

Statistical theory for  time series analysis is based on the crucial assumption of stationarity which supports
development of appropriate asymptotic theory. It is well understood now that the assumption of stationarity  is often unrealistic  and is not  applicable to real data problems originating   from  finance, neuroscience or environmental studies, to mention a few. Non-stationarity is  observed in several  ways. For example, a simple time series plot  might  reveal  significant fluctuations  due to periodicity and/or external factors that influence  the observations. 
From a theoretical perspective, lack of stationarity poses important issues associated  with model identification and prediction. Furthermore,  asymptotic analysis is unsupported because any future information of the process does not necessarily provide  knowledge  about its  current state.  
These  problems persist  in applications.

The goal of this work is to develop properties and inference for non-stationary processes with heavy-tailed innovations. An appropriate time series framework for this development is given by the notion of local stationarity, as introduced by 	\citet{dahlhaus1996kullback}. This framework has inspired several works  which have advanced theory and methodology for non-stationary time series but under the assumption of finite variance for the innovation sequence. To the best of our knowledge, the  important case of locally stationary processes 
with an  infinite variance innovation sequence has been neglected in the literature.   In this contribution, we fill 
this gap by introducing an infinite order time-varying moving average (MA) process with heavy-tailed innovations  whose coefficients change smoothly  over time.

\subsection{Related work}
There exists a vast literature on models for heavy-tailed time series; see  {\cite{Embrechtsetal(1997)}}, \cite{Resnick(2007)} and   \cite{KulikandSoulier(2020)} for comprehensive reviews.  In the context of AR models with heavy-tailed innovations, \cite{davis1985limit, Davis1986a}  established the limiting distribution of the 
least squares estimator (LSE). Using point processes techniques, these authors proved  that the
LSE  converges weakly to a ratio of two stable random variables with the
rate $n^{1/\alpha}L(n)$, where $L(\cdot)$ is a slowly varying function, $n$ is the sample size and $\alpha \in (0,2)$ is the tail index of the innovation sequence.  Moreover, \cite{Davisetal(1992)}  developed asymptotic theory for 
$M$  and least absolute deviation estimators (LADE)  by proving that they converge weakly to the minimum of a stochastic process. The complicated form of the limiting distribution motivated \cite{Ling2005} to introduce 
a self-weighted LADE  and prove its asymptotic normality; see also \cite{Panetal(2007)} who study the  self-weighted LADE to the case of infinite variance autoregressive moving average (ARMA) models and show that the resulting estimators are asymptotically normally distributed.  In fact, various weighting schemes have been proposed over the years {to bypass issues associated with heavy tailed  data. All this methodology provides  appropriate choices  which} give consistent estimators  for the regression parameters and their respective asymptotic covariance matrix.  Some further recent related  work on  ARMA models 
with unspecified  heavy-tailed heteroscedastic noise is given by \cite{ZhuandLing(2015)}; see also \citet{Akashi(2017), Akashietal(2018)}, among others.
%This idea will be employed in the framework of locally stationary  processes which is a paradigm of non-stationary processes. 

%The multitude of data structures observed in applications makes the generalization of stationarity to non-stationarity implausible. 
The  first  attempt to deal with the problem of non-stationarity  dates back to  \cite{priestley1965evolutionary} who 
considered processes  whose characteristics are changing slowly over time and developed the theory of ``evolutionary spectra''; see also \cite{priestley1988non}.
However, such an approach does not support   asymptotic theory, which is needed for studying  estimation 
in this context. Other early approaches include the works by 	\cite{granger1964spectral} and 
\cite{tjostheim1976spectral}. 
To apply  asymptotic theory  to  non-stationary processes, Dahlhaus, in a series of contributions, introduced an appropriate theoretical framework, based on the concept of local stationarity (see, for example,
\citet{dahlhaus1996kullback, Dahlhaus(1997),Dahlhaus(2000)}). The  definition of local stationarity relies  on the existence  of a time-varying spectral representation (\cite{dahlhaus1996kullback}).  Such representation though does not hold in the case we study. Basically,
local stationarity means that a process, over short periods, behaves ``approximately'' as a stationary process; see 
\cite{dahlhaus2012locally} for more. 

Locally stationary models
have been mainly studied  within a parametric context where {by utilizing smooth time-varying models}; see \citet{DahhausandRao(2006)} who analyze a class of ARCH (AR conditionally heteroscedastic) models with time-varying  coefficients.  Besides developing appropriate stability conditions, these authors have proposed kernel-based quasi-maximum likelihood estimation methods for non-parametric estimation of time-varying parameters.  Related works include that of  \cite{Frylewiczetal(2008)} who study
a kernel-based normalized LSE, \cite{HafnerandLinton(2010)} who provide estimation theory for a
multivariate GARCH model with a time-varying unconditional variance and \citet{Truquet(2017)} who 
develops semiparametric inference for time-varying ARCH models. 
Nonlinear autoregressive models with time-varying coefficient have been studied in detail by \cite{Vogt(2012)}, see also
\citet{Truquet(2019)} for locally stationary Markov models and \citet{Bardetetal(2022)} for infinite-memory time series models. 
A diffusion process with a time-dependent drift and diffusion function is investigated
in \cite{KooandLinton(2012)} and  \cite{DelemontandLaVecchia(2019)} study M--estimation  method for locally stationary diffusions observed at discrete time-points.

Additional references for locally stationary time series include the works by
\cite{neumann1997wavelet},
\cite{nason2000wavelet},
\cite{ombao2002slex},
\cite{sakiyama2004discriminant}, 
\cite{davis2006structural}, \cite{ZhouandWu(2009)}, \cite{Detteetal(2011)},  \cite{Zhou(2014)}, \cite{Giraudetal(2015)},
\cite{ZhangandWu(2015)},  \cite{WuandZhou(2017)}, \cite{MoysiadisandFokianos(2017)},
\cite{Detteetal(2019)}
among others.  The paper by   \cite{Dahlhausetal(2019)}
provides  general theory for  locally stationary processes by studying in detail   the stationary approximation and its derivative; see also  \cite{DahlhausandGiraitis(1998)} and 
\cite{RichterandDahlhaus(2019)} for works related  the problem of bandwidth selection.

\subsection{Contribution and outline}

%Previous works have developed theory for locally stationary processes by imposing  finite moments assumption. 	
%This work removes this assumptions and it  is  making  several contributions when the innovation sequence does not posses any moments. %More precisely, the paper contains the following results:
%\begin{enumerate}  
Section \ref{Sec:tv-MA processes} introduces time-varying MA($\infty$) models with heavy-tailed innovations.
% and states the necessary assumptions for local approximation of such processes. 
One of the main results of this section is given by Proposition \ref{Prop: stationary approx} which provides a local approximation of a time-varying process by a stationary process.
An additional approximation is proved  which is instrumental  for implementing the weighting scheme used for estimation. Finally, Proposition \ref{Prop:equivalenve AR to MA} shows the existence of time-varying ARMA (tvARMA) models with heavy tailed innovations.  Section \ref{Sec:LocalPolynomialEstimation} focuses on tvAR($p$) models and develops estimation of the regression coefficients. Using the self-weighted scheme--which is a key idea in this study, Theorem \ref{thm:1} shows that the proposed method of estimation provides an asymptotically normally distributed estimator. The proofs of these results are based on Lemmas \ref{lem:3} and \ref{lem:4} which eventually imply that the proposed estimator 
	converges weakly to the minimum of a stochastic process. In addition, we propose a multiplier bootstrap methodology 
	which  provides a consistent estimator of the asymptotic variance of the estimator obtained by Theorem \ref{thm:1}; see Theorem \ref{thm:BS}. 
The paper concludes with a limited simulation study and a data analysis example;  see Sections \ref{sec:Simulations} and \ref{sec:Data analysis}. Finally, the %\erase{Supplement}
{Appendix}  contains the proofs of results proved throughout the paper.

\paragraph{Notation} Throughout this paper, $0_k$, $O_{k\times l}$ and $I_k$ denote, respectively, a $k$-dimensional zero vector, a $k\times l$-dimensional zero matrix and a $k$-dimensional identity matrix.
The set of all integers, positive integers and real numbers are denoted by $\Z$, $\N$ and $\R$, respectively. The notation 
$\N_{0}$ denotes the set of non-negative integers. 
For any sequence of random vectors $\{A(t): t\in \Z\}$, $A(t)\plim A$ and $A(t)\dlim A$ denote the convergence to a random (or constant) vector $A$ in probability and law, respectively.
The transpose and the conjugate transpose of any matrix $M$ are denoted by $M^\top$ and $M^*$, respectively.
For any vector $x$, its Euclidean norm is denoted by $\|x\|$ and for any matrix $M$, we define $\| M\| = \sqrt{\mathrm{tr}(M^\top M)}$.
{$\intp[x]$ denotes the integer part of a real number $x$.} The notation $\mathrm{sign}(x)$ represents the sign function and $\mathbb{I}(A)$ {denotes the indicator function of $A$. 

\section{ Local stationary  approximation of time-varying MA($\infty$) models with heavy-tailed innovations}
\label{Sec:tv-MA processes}

In what follows, \blue{we call a function $L:(0,\infty)\to(0,\infty)$ slowly varying if 
$\lim_{s\to\infty}L(sx)/L(s) = 1$ for all $x>0$.} Assume that	
\blue{
\begin{assumption}
	\label{ass:1}
	\begin{enumerate}[label = (\roman*)]
		\item $\{\epsilon_t:t\in\Z\}$ is a sequence of i.i.d.\ random variables;
		\item there exists $\alpha\in(0,2)$ and a slowly varying function $L(x)$ such that $P(|\epsilon_t|>x) = x^{-\alpha}L(x)$; 
		\item the median of %\erase{$e_t$} 
		$\epsilon_{t}$  is zero.
	\end{enumerate}
%	The innovation sequence is denoted by $\{\epsilon_t:t\in\Z\}$ and it is assumed to be  a strictly stationary process such  that
%	$P(|\epsilon_{t}|>x) = x^{-\alpha}L(x)$ with $\alpha>0$ and $L(x)$ a slowly varying function at $\infty$.
\end{assumption}
}
It is well known that, if  $\{c_j: j\in\N\}$ is  a sequence  of real numbers satisfying
$\sum^{\infty}_{j=0} |c_j|^\delta < \infty$,   for some $\delta\in(0,\min\{1,\alpha\})$, then
%\footnote{For example, it is shown that $c_j = j^a \rho^j$ ($a\in\R$, $|\rho|<1$) satisfies this condition by d'Alembert's ratio test.}
\blue{Assumptions \ref{ass:1}(i) and (ii)} imply that  the series
$\blue{A_t}:=\sum^{\infty}_{j=0}c_j|\epsilon_{t-j}|$ exists a.e.\ (see  \cite{cline1983estimation}, \cite{Davis1986a}). This fact will be used several times. 
For example, we bound higher-order terms of the expansion of some statistics by employing  $\sum^{n}_{t=p+1}A_t$.
Calculation of the stochastic order of $\sum^{n}_{t=p+1}A_t$ relies on  Theorem 4.1 of \cite{davis1985limit}. Note also that condition (1.2) of \cite{davis1985limit} is automatically satisfied for $Z_t = |\epsilon_t|$. 
Assumption \ref{ass:1}(iii) 
provides  asymptotic unbiasedness of the estimator.
Following \citet{Dahlhaus(1997), Dahlhaus(2000)}, define a non-stationary linear process by  
\begin{align}
	Y_{t,T} = \sum_{j=0}^{\infty} \psi_{j,t,T} \epsilon_{t-j},
	\label{def: Def nonstationary process}
\end{align}
where $\{\epsilon_t:t\in\Z\}$ is a process satisfying  Assumption \ref{ass:1}. We call  $\{Y_{t,T} : t=1,\ldots,T\}$
defined by \eqref{def: Def nonstationary process},  a time-varying MA($\infty$) process with heavy-tailed innovations.
We  show that \eqref{def: Def nonstationary process}
is approximated by  the stationary process
\begin{align}
	Y_{t}(u) = \sum_{j=0}^{\infty} \psi_{j}(u) \epsilon_{t-j},
	\label{def: Def stationary process}
\end{align}
for \blue{each} $u \in (0,1)$ such that $| t/T-u | < T^{-1}$ and $\psi_{j}(\cdot):$  $[0,1] \rightarrow \mathbb{R}$. 
This implies that the processes $Y_{t,T}$ and $Y_{t}(u)$ behave similarly  provided that $t/T$ is close to $u$ and the approximation should
depend on $T$ and $| t/T-u |$.  This construction allows rescaling
the time-varying parameters  $\psi_{j,t,T}$ to the unit interval $[0,1]$ assuming that there exist functions $\psi_{j}(\cdot): [0.1]\rightarrow \mathbb{R}$
such that  $\psi_{j,t,T} \approx \psi_{j}(t/T)$, in some sense. 
The reasons for the rescaling are described in detail, e.g. in
\citet[Sec.2]{dahlhaus2012locally}. To prove the validity of the local  approximation we assume the following conditions.

\begin{assumption}\label{ass:2} 
	There exists a sequence $\{l_j: j \in \N_{0}\}$ such that
	$\sum^{\infty}_{j=0}(l_j^{-1})^\delta<\infty$ for some $\delta\in(0,\min\{1,\alpha\})$,
	where $\alpha$ has been defined in Assumption \ref{ass:1}.
	Moreover, $\{\psi_{j,t,T}: j\in\N_{0}, t=1,...,T\}$ and \blue{$\{\psi_{j}(u): j\in\N_{0}\}$}
	satisfy the followings conditions:
	(i) $\sup_{t=1,...,T} |\psi_{j,t,T}| \leq  \blue{C_0}/l_j$, \blue{$|\psi_{j}(u)| \leq  \blue{C_0}/l_j$};
	(ii) $\sup_{t=1,...,T} \left|\psi_{j,t,T} - \psi_{j}(t/T)\right|\leq \blue{C_0}/(T l_j)$; 
	$|\psi_j(u)-\psi_j(v)|\leq \blue{C_0}|u-v|/l_j$,
	where $\blue{C_0}$ is a positive constant which does not depend of $T$.
\end{assumption}

To develop estimation, we further study 
a non-stationary  \textit{self-weighting} process defined by    $w_{t-1,T}:=g(X_{t-1,T})$,
where $g: \R^p\to\R^+$ is a positive, measurable function of 
$X_{t -1,T} = (Y_{t -1,T},...,Y_{t-p,T})^\top$. Analogously, let us  denote by
$w_{t-1}(u):=g(X_{t-1}(u))$ with
$X_{t -1}(u) = (Y_{t -1}(u),...,Y_{t -p}(u))^\top$, its corresponding stationary approximation, assuming that it exists.
%\erase{The self-weighting process is introduced for the development of  robust estimation.}  
%Given  a suitable weight function $g(\cdot)$, we implement non-parametric estimation of a  regression model by
%following the methodology of \citet{Ling2005}.   
We assume that the function $g(\cdot)$ satisfies the following assumption:

\begin{assumption}\mbox{}\label{ass:3} 
	$g: \R^p\to\R^+$ is  a differentiable  function that satisfies 
	$M := \sup_{x\in\R^p}  \bigl[g(x)(1+ \|x\|^3) +   \|g'(x)\| (\|x\| + \|x\|^2) \bigr]<\infty$.
\end{assumption}

Some examples of  weight functions $g(\cdot)$ satisfying Assumption \ref{ass:3} are given below.  
Assume that $x \in \R^p$.
	\begin{example}
		\label{eg:1}
		Define
		\begin{align}
			g(x) := \frac{1}{(1+c\|x\|^2)^{3/2}},\label{eq:w1}.
		\end{align}
		This weight function was proposed by \cite{Ling2005} and it holds that \blue{$g(x) \leq 1$} for all $x$.
		As $\|x\|\to\infty$, we obtain  the  approximation $g(x) \sim (\sqrt{c}\|x\|)^{-3}$.
		Then, \blue{$g(x)\|x\|^3$} is bounded by a finite constant uniformly in $x\in\R^p$.
		It is also   easy to see that  \blue{$\|g'(x)\|(\|x\|+\|x\|^2)$} is bounded 
		by a constant uniformly in $x\in\R^p$. Hence   Assumption \ref{ass:3} is satisfied.
	\end{example}
	
	\begin{example}\label{eg:2}
		A natural example of weight function is the indicator  function $\mathbb{I}(\|x\|<c)$.
		However, this  function does not satisfy Assumption \ref{ass:3}
		since it is not differentiable at the point $x$ with  $\|x\|=c$. 
		To address this issue, we replace the indicator  by a  smooth approximation. Specifically, let 
		\begin{align}
			g(x) = J(c-\|x\|),\label{eq:w3}
		\end{align}
		where
		\begin{align}
			&J(u) = \begin{cases}
				0, & (u\leq -1),\\
				-0.25u^3 + 0.75u + 0.5, & (-1<u\leq 1),\\
				1, & (1<u).
			\end{cases}\notag
		\end{align}
		{Assuming $c \geq 1$, the resulting weight function satisfies  Assumption \ref{ass:3}. Alternatively, one may use   $g(x)=J(c-\| x\|^2)$, which is differentiable everywhere for any $c >0$   and also satisfies  Assumption \ref{ass:3}.}  Additional examples of weight functions can be found in  \citet{Ling2005} and \citet{Panetal(2007)}. 
\end{example}

The main result of this section is the following proposition whose proof is postponed to the appendix.
\begin{proposition}
	\label{Prop: stationary approx}
	\begin{enumerate}
		\item  Suppose that Assumptions \ref{ass:1} and \ref{ass:2} hold true \blue{for $u=u_0\in(0,1)$}.
		Assume that $\{Y_{t,T}: t=1,\ldots, T \}$ is a time-varying MA($\infty$) process defined by \eqref{def: Def nonstationary process}
		and let $ \{Y_{t}(u): t=1,\ldots,T\}$ be defined as in \eqref{def: Def stationary process}. Then both processes are well defined
		and if \blue{$t$ satisfies $| t/T-u_0 | < T^{-1}$}, then
		\begin{align*}
			| Y_{t,T}- Y_{t}\blue{(u_0)}| \leq  \blue{C_0}\left(\left| \frac{t}{T}-\blue{u_0}\right|+\frac{1}{T} \right)  \sum_{j=0}^{\infty} \frac{1}{l_{j}} |\epsilon_{t-j}|.
		\end{align*}
		\item In addition, suppose that Assumption \ref{ass:3} holds.
		Then, if $\{w_{t-1,T}\}$ and $\{w_{t-1}(u_0)\}$ are defined
		as above, then for \blue{$t$ satisfying $| t/T-u_0 | < T^{-1}$} and $M>0$ constant,
		\begin{align*}
			\|w_{t-1,T}X_{t-1,T} - w_{t-1}\blue{(u_0)} X_{t-1}\blue{(u_0)}\| \leq \left(
			\left|\frac{t}{T}-\blue{u_0}\right| + \frac{1}{T}
			\right){ 2\blue{C_0}M}\sum^{p}_{i=1}\sum^{\infty}_{j=0} \frac{1}{l_{j}}|\epsilon_{t-i-j}|.
		\end{align*}
	\end{enumerate}
	
\end{proposition}

The last result of this Section proves  the  link between  the  non-stationary model \eqref{def: Def nonstationary process} and  tvARMA 
process; see   \citet[Prop. 2.4]{DahlhausandPolonik(2009)} and \citet{Giraudetal(2015)}. Following the discussion of  \citet{DahlhausandPolonik(2009)}, 
tvARMA models can be represented in terms
of the non-stationary process \eqref{def: Def nonstationary process} but they cannot be expressed in the form of
``$Y_{t, T} = \sum_{j=0}^{\infty} \psi_{j}(t/T) \epsilon_{t-j}$''.
\blue{For example, consider the tvAR(1) model $Y_{t,T} = \beta_1(t/T)Y_{t-1,T} + \epsilon_t$ with $\sup_{u\in[0,1]}|\beta_1(u)|<1$. By following \cite{kunsch1995ar}, \cite{DahlhausandPolonik(2009)} and \citet[Proposition 13.3.1]{BrockwellandDavis(1991)}
	with Assumptions \ref{ass:1} and \ref{ass:2}, we obtain 
	\begin{align}
		Y_{t,T}=\sum^{\infty}_{j=0}\left(\prod^{j-1}_{k=0}\beta_1\left(\frac{t-k}{T}\right)\right)\epsilon_{t-j}.\label{eq:X_tT}
	\end{align}
	Hence,  $\psi_{j,t,T} = \prod^{j-1}_{k=0}\beta_1((t-k)/T)$ are  not represented in the form of $\psi_l(t/T)$.
	In other words, $Y_{t,T}$ is not exactly represented as the right hand side of \eqref{def: Def stationary process} in general, but is approximated by a stationary process $Y_t(u)$ in a certain sense.}
To prove the result we employ again Assumption \ref{def: Def nonstationary process}.
%\erase{but under the hypothesis that $\{ \epsilon_{t} \}$ is an iid sequence.}
\begin{proposition}\label{prop:2}
	\blue{Suppose that $\{\epsilon_t: t\in\Z\}$ satisfies Assumption \ref{def: Def nonstationary process}. }
	Consider the tvARMA model defined by
	\begin{align*}
		\blue{Y_{t,T} = \sum_{j=1}^{p} \beta_{j}\left(\frac{t}{T}\right) Y_{t-j, T} + \sum_{k=0}^{q} \gamma_{k}\left(\frac{t}{T}\right)\epsilon_{t-k}},
	\end{align*}
	where %\erase{$\{ \epsilon_{t} \}$ is an iid sequence satisfying Assumption \ref{ass:1},} 
	$\gamma_{0}(u)=1$ and for $u <0$, $\beta_{j}(u)=\beta_{j}(0)$ for $j=1,2\ldots, p$ and
	$\gamma_{k}(u)=\gamma_{k}(0)$, $k=1,2,\ldots, q$. Suppose that all functions $\beta_{j}(\cdot)$ and $\gamma_{k}(\cdot)$ are Lipschitz and $\sum_{j=0}^{p} \beta_{j}(u) z^{j} \neq 0$ for
	all $u$ and $0 < |z| \leq  1+\delta$ for a positive $\delta$. Then there exists solution of the form \eqref{def: Def nonstationary process} which satisfies the conditions of Assumption \ref{ass:2}.
	\label{Prop:equivalenve AR to MA}
\end{proposition}

\section{Estimation for tvAR($p$) models with heavy-tailed  innovations}
\label{Sec:LocalPolynomialEstimation}
In what follows, we assume  that the  observed  time series data  $\{Y_{1,T},...,Y_{T,T}\}$
is generated by the following tvAR($p$) model
\begin{align}
	Y_{t,T}  = \sum_{j=1}^{p}\beta_j\left(\frac{t}{T}\right)Y_{t-j,T} + \epsilon_{t},\label{ar}
\end{align}
where $\{ \epsilon_{t} \}$ is a sequence of iid random variables %\erase{with zero median such that it  satisfies}
satisfying   Assumption \ref{ass:1}.
Based  on Proposition \ref{Prop:equivalenve AR to MA}, the model \eqref{ar} admits the invertible MA($\infty$) representation \eqref{def: Def nonstationary process}.
In this section we construct and study estimators for the unknown coefficient function $\{ \beta_{j}(\cdot), j=1,\ldots, p\}$ under the above setup.
%by appealing to the methodology of local regression; see \citet{FanandYao(2003)}, for instance.

\subsection{{Local regression estimator}}
Set $\beta(\cdot)=(\beta_{1}(\cdot), \ldots, \beta_{p}(\cdot))^\top$ and define the residual process
$e_{t-1,T}(\beta) := Y_{t,T} - \beta^\top X_{t-1,T}$ ($t=p+1,...,T$),
and  consider \textit{{Local} Self-Weighted Least Absolute Deviation Estimator (LSWLADE)} at {$u=u_0$} defined as
\begin{align}
	{\hat\beta_{T}(u_0)} := \arg\min_{\beta} {
		\sum^{T}_{t= p+1}
		K_h(t-\intp[u_0T])
		w_{t-1,T} |e_{t-1,T}(\beta)|},
	\label{local_SLADE}
\end{align}
by utilizing   the self-weight function  defined by $w_{t-1,T}:=g(X_{t-1,T})$; recall  Section \ref{Sec:tv-MA processes}. 
Furthermore  {$K_h(t-\intp[u_0T])=h^{-1}K((t-\intp[u_0T])/(Th))$}  with $K(\cdot)$ being  a kernel function, $h$ is a bandwidth parameter, which converges to zero as $T\to\infty$. 
We prove  that this particular weighting scheme yields consistent estimators for the regression parameters and their respective asymptotic covariance matrix. If $w_{t-1,T} \equiv 1$, then \eqref{local_SLADE} is simply a version of  
least absolute deviation estimator (see \citet{DavisandDunsmuir(1997)}) but  for locally stationary 
processes;  see \cite{Halletal(2002)}. Additionally, when  $w_{t-1,T} \equiv 1$, \eqref{local_SLADE} is simply the smooth conditional likelihood 
approximation to the case of conditionally Laplace distributed time series.

%\subsection{\erase{Examples of weight functions}}

\subsection{Asymptotic inference}

%Suppose that the tv-AR model \eqref{ar} admits the invertible MA($\infty$) representation
%\begin{align}
%Y_{t,T} = \sum_{j=-\infty}^{\infty} \psi_{j,t,T} \epsilon_{t-j},
%\label{def: Def nonstationary process}
%\end{align}
%where suitable assumption on the time-varying functions $\psi_{j,t,T}$ will be given below to ensure that \eqref{def: Def nonstationary process} is well-defined.
%The main aim is to show that \eqref{def: Def nonstationary process}
%is approximated by  the stationary process
%\begin{align}
%Y_{t}(u) = \sum_{j=-\infty}^{\infty} \psi_{j}(u) \epsilon_{t-j},
%\label{def: Def stationary process}
%\end{align}
%for all $u \in (0,1)$  and $\psi_{j}(\cdot)$ are differentiable functions $[0,1] \rightarrow \mathbb{R}$.
%We also define $w_{t-1}(u):=g(Y_{t-1}(u))$, where
%$X_{t -1}(u) = (Y_{t -1}(u),...,Y_{t -p}(u))^\top$.
%For the notational convenience, we show the following results with $p=1$, but the extension for general $p\geq2$ is quite similar.
%We will be using the following conditions:

%Under the assumptions, we get the following lemma.
%
%\begin{lemma}\label{lem:1}
%Under Assumptions \ref{ass:1} and \ref{ass:2},
%\begin{align}
%\|w_{t-1,T}X_{t-1,T} - w_{t-1}(u) X_{t-1}(u)\| \leq \left(
%\frac{1}{T} + \left|\frac{t}{T}-u\right|
%\right)L\delta\sum^{p}_{i=1}\sum^{\infty}_{j=0}l_j^{-1}|\epsilon_{t-i-j}|.
%\end{align}
%\end{lemma}
%The proof of Lemma \ref{lem:1} is relegated to the section of Proof.

We  derive the limit distribution of  the LSWLADE
$\hat\beta_{T}(u_0)$ defined by \eqref{local_SLADE}.
%First, as in Proposition \ref{Prop:equivalenve AR to MA}, we assume that the error sequence $\{ \epsilon_{t} \}$
%consists of iid random variables satisfying Assumption \ref{ass:1}. 
The following  conditions are mild and %have been in 
used %the literature 
to establish large-sample behavior of  non-parametric estimators.

\begin{assumption}\label{ass:n4}
	%\erase{For $\alpha \in (0,2)$, defined in Assumption \ref{ass:1}, suppose} 
	Suppose that 
	\[
	\frac{h}{T}\blue{a_{\intp[Th]}^2}\to 0
	\text{ and }
	Th^3 \blue{b_{\intp[Th]}}^2 \to 0,
	\]
	where \blue{$a_n$ and $b_n$ ($n\in\N$) are sequences of positive real numbers defined as} 
	$a_n := \inf \left\{x: P(|\epsilon_1|>x)\leq 1/n\right\}$
	and 
	$b_n := \E\left[ |\epsilon_1| \mathbb{I}(|\epsilon_1|\leq a_n) \right]$ $(n\in \N)$.
\end{assumption}
\noindent
Assumption \ref{ass:n4} is required for proving an appropriate  approximation of the objective function obtained by the   
non-stationary data  by the  corresponding  objective function  
based on their  stationary version. This is a mild assumption and it is easy to see that is satisfied when $\epsilon_1$ follows, for example,  the standard Cauchy distribution.
In this case,  a simple calculation shows that 
\[
a_n = \frac{2n}{\pi} + o(n)\quad \text{ and }\quad
b_n = \frac{\log(1+a_n^2)}{\pi}.
\]
Assumption \ref{ass:n4} is fulfilled if $Th^3(\log(Th))^2\to 0$.
In this case, a sufficient conditions for $h$ is $\limsup_{T\to\infty}Th^{\delta}<\infty$ for some $\delta\in(1,3)$.

\medskip

More generally, suppose that there exists finite positive constants\ $K, c_1$ and $c_2$ such that  the function $L$ of  Assumption \ref{ass:1} satisfies 
$0< c_1\leq L(x) \leq c_2<\infty$ for all $x\geq K$.
Then, 
for sufficiently large $n$, 
it can be  shown that
$c_1'n^{1/\alpha}\leq \blue{a_n}\leq 
c_2'n^{1/\alpha}$, where $c_1'$ and $c_2'$ are finite absolute constants.
Thus, $a_n$ is of order $n^{1/\alpha}$.
Hence, the first condition of Assumption \ref{ass:n4} \blue{is equivalent to}
\[
\frac{h}{T}(Th)^{2/\alpha} = T^{2/\alpha -1}h^{2/\alpha+1}\to 0.
\]
%\erase{If $\alpha\geq2$, then any $h\downarrow0$ satisfies the condition.}
If $0<\alpha<2$, then $h$ should {be of the order}  $o(1/T^{(2-\alpha)/(2+\alpha)})$, 
\blue{where the exponent $(2-\alpha)/(2+\alpha)$ belongs to the interval $(0,1)$}.
To derive the order of $b_n$, we consider an additional structure of the model. 
Suppose that the density function of $\epsilon_t$ is given by
\begin{align}
	f(y) = 
	\begin{cases}
		g_\alpha(y) & (|y|\leq C)\\
		l_\alpha(y) |y|^{-\alpha-1} & (|y|> C)
	\end{cases},\notag
\end{align}
where $g_\alpha$ and $l_\alpha$ are chosen so that $f$ satisfies the properties of a density function, and assume that there exist finite positive constants $c_3$ and $c_4$ such that
$c_3\leq l_\alpha(x)\leq c_4$ for all $|y|> C$. This construction  includes the Cauchy distribution with $\alpha=1$, $l_\alpha(x) := x^2/(\pi(1+x^2))$, $C=1$, $c_3=1/(2\pi)$ and $c_4 = 1$.
In this case, for sufficient large $n$, some simple calculations yield
that the order of $b_n$ is $O(1)$ ($\alpha>1$),
$O(\log n)$ ($\alpha=1$) and $O(n^{1-\alpha})$ ($0<\alpha<1$). Summarizing
\begin{itemize}
	\item If $0<\alpha<1$, then $Th^{(2+\alpha)/(2-\alpha)}, Th^{(5-2\alpha)/(3-2\alpha)}\to0$.
	Since $(2+\alpha)/(2-\alpha)<(5-2\alpha)/(3-2\alpha)$,
	the condition  $Th^{(2+\alpha)/(2-\alpha)}\to 0$ suffices, where $(2+\alpha)/(2-\alpha)\in(1,3)$.
	\item If $\alpha=1$, then $(Th^3)^{1/2}\log(Th)\to 0$.  A sufficient conditions for this to hold  is 
	$\limsup_{T\to\infty}Th^{\delta}<\infty$ for some $\delta\in(1,3)$.
	\item If $1<\alpha<2$, then $Th^3\to0$.
\end{itemize}

\begin{assumption}\label{new_ass:5}\mbox{} 
		(i) 	The density function $f$ of $\epsilon_1$ satisfies $f(0)>0$ and $\sup_s \{|f'(s)| + f(s)\}<\infty$;
		(ii)	The matrices $\Sigma(u):=\E[w_{t-1}(u)X_{t-1}(u)X^\top_{t-1}(u)]$ and
		$\Omega(u):=\E[w_{t-1}^2(u) X_{t-1}(u)X^\top_{t-1}(u)]$ are non-singular; 
		%\erase{for all $u\in[0,1]$};
		(iii)	The kernel function $K$ is symmetric about zero, bounded and 
		$K(x)=0$ for $|x|>C_K$ with some $C_{K} < \infty$. 
		(iv)	$\kappa_2 := \int^{C_K}_{-C_K}K^2(v)dv>0$. 
\end{assumption}

\begin{theorem}\label{thm:1} 
	Consider the tvAR($p$) process defined by \eqref{ar} and let ${\hat\beta_T(u_0)} $ defined by \eqref{local_SLADE}. Suppose further  that Assumptions \ref{ass:1}--\ref{new_ass:5} hold true \blue{for $u=u_0$} and that the bandwidth parameter  $h$ satisfies
	$Th\to\infty$, $Th^3 \to 0$ as $T\to\infty$.
	Then,
	\begin{align}
		\sqrt{Th}\left[{\hat\beta_T(u_0)} - \beta(u_0)\right]
		\dlim N\left(
		0_p, \frac{\kappa_2}{4f^2(0)} \Sigma^{-1}(u_0)\Omega(u_0)\Sigma^{-1}(u_0)
		\right). \label{eq:lim_thm1}
	\end{align}
\end{theorem}

\medskip

Theorem \ref{thm:1}
generalizes \citet[Thm.1]{Ling2005} and \citet[Thm. 2]{Panetal(2007)} in the framework of locally stationary processes with
suitable weight functions. Its proof is postponed to the Appendix. 
In addition, it  provides pointwise asymptotic results. It is interesting to extend this work by showing uniform convergence of the estimators and develop associated testing theory. This will be reported elsewhere.

\subsection{Bootstrap estimator of asymptotic variance}
The asymptotic variance of the LSWLADE is composed of  the unknown quantities $\Sigma(u_0)$, $\Omega(u_0)$ and $f^2(0)$.
The first two matrices are estimated by their sample analogue; that is, if we define 
\begin{align}
	\hat{V}_T^{(j)} := \frac{1}{Th}\sum_{t=p+1}^T K\left(\frac{t-{\intp[u_0T]}}{Th}\right) w_{t-1,T}^jX_{t-1,T}X_{t-1,T}^\top, \notag
\end{align}
then $\hat{V}_T^{(1)}$ and $\hat{V}_T^{(2)}$ are, respectively,  consistent estimators of $\Sigma(u_0)$ and $\Omega(u_0)$.
However, a kernel density %\erase{based} 
estimator for $f(0)$ is still needed and this  requires an additional tuning parameter. 
We  apply the multiplier bootstrap as in  \cite{ZhuandLing(2015)}. Define 
\begin{align}
	\hat\beta_{T}^*(u_0) := \arg\min_{\beta}
	\sum^{T}_{t= p+1}
	z_tK_h(t-\intp[u_0T])
	w_{t-1,T} |e_{t-1,T}(\beta)|, \label{BS_local_SLADE}
\end{align}
where $\{z_t: t=p+1,...,T\}$ is a sequence of {positive} i.i.d.\ random variables with unit mean and variance.
Repeating this minimization for $B$ times with independently drawn $\{z_t: t=p+1,...,T\}$, we obtain 
$B$-bootstrap estimators $\hat\beta_{T}^{*(1)}(u_0), ..., \hat\beta_{T}^{*(B)}(u_0)$.
By using these replicated statistic, we can construct\blue{, for example,} a confidence region for $\beta(u_0)$.
\blue{On the other hand,}
by a user-specified $q\times p$ matrix $R$ and $q\times 1$ vector $c$,
the Wald test statistic for the null hypothesis $H: R\beta(u_0) = c$ is defined as
\begin{align}
	W_n := (R\hat{\beta}_{T}(u_0) - c)^\top (R\hat{V}_{T,B}R^\top)^{-1}(R\hat{\beta}_{T}(u_0) - c),\notag 
\end{align}
where $\hat{V}_{T,B}$ is the sample covariance matrix of $\hat\beta_{T}^{*(1)}(u_0), ..., \hat\beta_{T}^{*(B)}(u_0)$.
Under the null hypothesis $H$, it holds that $W_n\dlim \chi^2_q$. 
The validity of these procedures is guaranteed by the following theorem.
\begin{theorem}\label{thm:BS}
	Suppose that all assumptions in Theorem \ref{thm:1} hold. In addition, suppose that $a_n$ in Assumption \ref{ass:n4} satisfies $h\blue{a_{\intp[Th]}}\to0$.
	Then, conditional on $Y_1(u_0), ..., Y_T(u_0)$, it holds that 
	\begin{align}
		\sqrt{Th}\left(\hat\beta_{T}^*(u_0) - \hat\beta_{T}(u_0)\right)
		\dlim N\left(
		0_p, \frac{\kappa_2}{4f^2(0)} \Sigma^{-1}(u_0)\Omega(u_0)\Sigma^{-1}(u_0)
		\right)\quad(T\to\infty)\notag
	\end{align}
	{in probability.}
\end{theorem}

\begin{remark}
	In the Cauchy case, the additional condition for $a_n$ is equivalent to $Th^2\to0$. 
\end{remark}

\begin{remark}
	The block bootstrap \citep{kunsch1989jackknife, doi:10.1080/01621459.1994.10476870} and 
	the dependent wild bootstrap \citep{shao2010dependent} 
	are well known resampling methods applied to time series data. 
	For infinite variance models, \cite{cavaliere2016sieve} proposed a sieve-based bootstrap method.
	%In general, these  methods are designed for stationary processes and might not reproduce the time-varying structure of nonstationary observations. To the best of our knowledge, there is no result of any time series resampling method related to  time-varying coefficient models. 
	The multiplier bootstrap estimator incorporates the local information by the kernel function, and the proposed method works for both stationary and non-stationary models.
\end{remark}

\section{Simulations}
\label{sec:Simulations}

We examine the finite sample performance of the estimator LSWLADE defined by \eqref{local_SLADE}. We use   the Epanechnikov  kernel throughout which is defined by 
$K(u)=0.75(1-u^{2})$, for $|u| \leq 1$, and apply  the weight functions \eqref{eq:w1} and \eqref{eq:w3}.
For the weight function \eqref{eq:w1}, we choose  $c$ as $0.5$ and $0.1$. 
This choice reflects two different real data situations where we want to ignore less ($c=0.1$) or more ($c=0.5$) amount of data. 
For the weight function \eqref{eq:w3}, we choose $c=\max(1,q_Y(0.90))$ and $c=\max(1,q_Y(0.95))$, where $q_Y(\delta)$ is the $\delta$th quantile of 
$\{|Y_{1,T}|,...,|Y_{T,T}|\}$
for an observed stretch $\{Y_{1,T},...,Y_{T,T}\}$.
%In addition, we will be using the weight function
%\begin{align}
%w_{t-1,T} := g(Y_{1,T}, \ldots, Y_{t-1,T}) = \Bigl(1+ \sum_{k=1}^{t-1} k^{-3} |Y_{t-k, T}| \Bigr)^{-2}
%\label{eq:weightsPenetal}
%\end{align}
%as suggested by \citet{Panetal(2007)} in the stationary case. 
The simulations are repeated 1000 times for each of the sample sizes {$T=100, 500$ and $1000$}.
Throughout this exercise, the bandwidth parameter $h$ is chosen as $h=\log(T)/T^{3/5}$. This particular choice is implied by the conditions of Theorems \ref{thm:1} and \ref{thm:BS}.

\subsection{Estimation results}

We first investigate the performance of  LSWLADE  by using different weighting schemes. This is  assessed by employing
the {Mean Absolute Error} of the estimator $\hat{\beta}(\cdot)$ for the parameter $\beta(\cdot)$, i.e. {
\begin{align}
\mbox{MSE} = \frac{1}{T} \sum_{i=1}^{T} \| \hat{\beta}(\frac{i}{T})-\beta(\frac{i}{T}) \|^2\label{mae}
\end{align}
}
Initially  consider the following tvAR(1) model
\begin{align}
Y_{t,T}= 0.8 \sin\left(\frac{2 \pi t}{T}\right) Y_{t-1, T}+ \epsilon_{t},
\label{eq:example1sim}
\end{align}
where the errors are assumed to be standard normal, $t$-distributed  with 2, {1.5}, and 1 (Cauchy)  degrees of freedom.% and 
%{the symmetric $\alpha$-stable distribution with $\alpha=0.5$}. 
{We consider the $t_{1.5}$ distribution because its tail exponent is consistent with that estimated from the data analyzed in the subsequent section.}
% The symmetric $\alpha$-stable case is also included to illustrate the broader applicability of the proposed methodology. However, as shown below, achieving accurate significance levels and coverage probabilities in this setting requires larger sample sizes. }
Under the aforementioned settings, we compare the {estimated MSE of the estimators based on the 1000 simulations using 
the overall MSE
\begin{align}
\widehat{\mathrm{MSE}} := \frac{1}{1000}\sum^{1000}_{j=1}\mathrm{MSE}_j, \label{eq:est_mae}
\end{align}
where $\mathrm{MSE}_j$ is obtained by the  $j$th simulation ($j=1,...,1000$).}
We compare the following estimators and the results are given  in Table \ref{tbl_MAE1}.
\begin{itemize}
\item[] L2: the unweighted $L_2$-estimator, which is minimizer of 
$\sum^{T}_{t= p+1}K_h(t-\intp[u_0T])|e_{t-1,T}(\beta)|^2$;
\item[] LAD: the unweighted least absolute deviations-estimator, which is minimizer of
$\sum^{T}_{t= p+1}K_h(t-\intp[u_0T])|e_{t-1,T}(\beta)|$;
\item[] LSW1c1: LSWLADE with the weight function \eqref{eq:w1} and $c=0.5$;
\item[] LSW1c2: LSWLADE with the weight function \eqref{eq:w1} and $c=0.1$;
\item[] LSW2q1: LSWLADE with the weight function \eqref{eq:w3} and $c=\max(1,q_Y(0.95))$;
\item[] LSW2q2: LSWLADE with the weight function \eqref{eq:w3} and $c=\max(1,q_Y(0.90))$.
%\item[] LSW3:   LSWLADE with the weight function \eqref{eq:weightsPenetal}.
\end{itemize}

\begin{table}[H]
\centering
\caption{Estimated MSE for the model \eqref{eq:example1sim} based on 1000 simulations and using different sample sizes. Minimum in each row is indicated by boldface fonts.}\label{tbl_MAE1}
\medskip
\begin{tabular}{lcccccc}\hline
	$\epsilon_t\sim N(0,1)$ & L2 & LAD & LSW1c1 & LSW1c2 & LSW2q1 & LSW2q2 \\\hline
	$T=100$     & \textbf{0.0656} & 0.0739 & 0.0821 & 0.0722 & 0.0754 & 0.0798 \\
	$T=500$     & \textbf{0.0154} & 0.0188 & 0.0225 & 0.0187 & 0.0197 & 0.0215 \\
	$T=1000$    & \textbf{0.0077} & 0.0100 & 0.0128 & 0.0102 & 0.0108 & 0.0120 \\\hline
%\end{tabular}\\
%\bigskip
%\begin{tabular}{lcccccc}\hline
	$\epsilon_t\sim t_2$ & L2 & LAD & LSW1c1 & LSW1c2 & LSW2q1 & LSW2q2 \\\hline
	$T=100$     & 0.0813 & 0.0732 & 0.0867 & \textbf{0.0620} & 0.0654 & 0.0671 \\
	$T=500$     & 0.0213 & 0.0177 & 0.0223 & \textbf{0.0153} & 0.0154 & 0.0165 \\
	$T=1000$    & 0.0111 & 0.0083 & 0.0122 & 0.0077 & \textbf{0.0074} & 0.0083 \\\hline
%\end{tabular}\\
%\bigskip
%\begin{tabular}{lcccccc}\hline
	$\epsilon_t\sim t_{1.5}$ & L2 & LAD & LSW1c1 & LSW1c2 & LSW2q1 & LSW2q2 \\\hline
	$T=100$	    & 0.0947 & 0.0803 & 0.0944 & \textbf{0.0621} & 0.0647 & 0.0644\\
	$T=500$	    & 0.0262 & 0.0207 & 0.0235 & 0.0150 &\textbf{ 0.0149} & 0.0154\\
	$T=1000$	& 0.0143 & 0.0101 & 0.0130 & 0.0076 &\textbf{ 0.0068 }& 0.0074\\ \hline
%\end{tabular}\\
	$\epsilon_t\sim \text{Cauchy}$ & L2 & LAD & LSW1c1 & LSW1c2 & LSW2q1 & LSW2q2 \\\hline
	$T=100$     & 0.1197 & 0.1015 & 0.1387 & 0.0694 & 0.0747 & \textbf{0.0668} \\
	$T=500$     & 0.0384 & 0.0295 & 0.0278 & 0.0154 & 0.0168 & \textbf{0.0149} \\
	$T=1000$    & 0.0210 & 0.0158 & 0.0151 & 0.0077 & 0.0071 & \textbf{0.0066} \\\hline
%\end{tabular}\\
%\bigskip
%\begin{tabular}{lcccccc}\hline
    %$\epsilon_t\sim \text{S$\alpha$S}$ & L2 & LAD & LSW1c1 & LSW1c2 & LSW2q1 & LSW2q2 \\\hline
	%$T=100$     & 200.6477 & 0.1506 & (*1) & (*2) & 0.1933 & 1.9866 \\
	%$T=500$     & 9.9991 & 0.0511 & 3.183 & 0.0324 & 0.0436 & \textbf{0.0249} \\
	%$T=1000$    & 0.2330 & 0.0286 & 1.4048 & 0.0190 & 0.0157 & \textbf{0.0127} \\\hline
	% (*1) 165765222684.8910
	% (*2) 91140958359.8860
\end{tabular}
%\begin{tabular}{lccccccc}
%$\epsilon_t\sim N(0,1)$ & L2 & LAD & LSW1c1 & LSW1c2 & LSW2q1 & LSW2q2 & LSW3\\\hline
%$T=100$ 		& \textbf{0.2063} & 0.2166 & 0.2293 & 0.2140 & 0.2193 & 0.2256 & 0.2227 \\
%$T=500$ 		& \textbf{0.0937} & 0.1034 & 0.1135 & 0.1025 & 0.1062 & 0.1114 & 0.1092 \\
%$T=1000$   	& \textbf{0.0659} & 0.0761 & 0.0864 & 0.0766 & 0.0798 & 0.0847 & 0.0829 \\\hline
%$\epsilon_t\sim t_2$&  &  &  &  &  &  &  \\\hline
%$T=100$ 		& 0.2279 & 0.2119 & 0.2376 & \textbf{0.2011} & 0.2047 & 0.2087 & 0.2101 \\
%$T=500$ 		& 0.1102 & 0.0986 & 0.1143 & \textbf{0.0926} & 0.0929 & 0.0975 & 0.0978 \\
%$T=1000$ 	& 0.0792 & 0.0686 & 0.0856 & 0.0667 & \textbf{0.0655} & 0.0703 & 0.0714 \\\hline
%$\epsilon_t\sim $Cauchy &  &  &  &  &  &  &  \\\hline
%$T=100$ 		& 0.2823 & 0.2456 & 0.2801 & 0.2114 & 0.2103 & \textbf{0.2017} & 0.2123 \\
%$T=500$ 		& 0.1461 & 0.1275 & 0.1272 & 0.0923 & 0.0955 & \textbf{0.0905} & 0.0927 \\
%$T=1000$ 	& 0.1080 & 0.0914 & 0.0940 & 0.0656 & 0.0614 & \textbf{0.0593} & 0.0660 \\\hline
%\end{tabular}
\end{table}
}
As expected, the $L_2$ estimator performs best when the underlying process is Gaussian. In contrast, the self-weighting scheme exhibits better  performance when the error terms follow heavy-tailed distributions.  {In particular, the LSWLADE estimator, computed using the weight function \eqref{eq:w3} with $c=\max(1,q_Y(0.90))$ achieves the lowest MSE when the errors follow a Cauchy distribution %or a  $\alpha$-stable distribution. 
For the $t$-distribution, LSW1c2 and LSW2q1 exhibit comparable performance and attain the smallest MSE among all competing estimators}. %Note that, for the symmetric  $\alpha$-stable distribution, results are reported only for $T=500,1000$. }

More generally, it would be of considerable interest to investigate whether a particular weighting function $g$  possesses optimality properties, for example in terms of minimizing the MSE. The primary objective of the present work is to demonstrate that appropriately chosen weighting schemes yield estimators with desirable large-sample properties. Determining an optimal weighting function $g$, should one exist, remains an open problem and constitutes an important direction for future research.

\subsection{Equivalence test of parameter values at different time points}\label{ssec:6.1}

We  provide an applications of Theorem \ref{thm:BS} to inference for time-varying models.
We first consider the AR(1) model $Y_{t,T}=  \beta(t/T)Y_{t-1, T}+ \epsilon_{t}$,
where $\beta(u) = 0.8 \sin(4 \pi u)$.
For user-specified points $u_1, u_2\in(0,1)$, let us focus on the testing problem of hypothesis 
\begin{align}
H_{0}: \beta(u_1) = \beta(u_2)\text{ against }
H_{1}: \beta(u_1) \neq \beta(u_2).\notag
\end{align}
This hypothesis testing considers whether the parameter values at $u_1$ and $u_2$ are different or not.
For this problem, we construct a test statistic as follows

By Theorem \ref{thm:1} and Cram\'er-Wold device, 
we can show that, under $H_0$, 
\begin{align}
\sqrt{Th}\left(\hat\beta_T(u_1) - \hat\beta_T(u_2)\right)
&=\sqrt{Th}\left(\hat\beta_T(u_1) - \beta(u_1)\right) - \sqrt{Th}\left(\hat\beta_T(u_2) - \beta(u_2)\right)\notag\\
&\dlim N(0, \sigma_1^2+\sigma_2^2 - 2\sigma_{12})
\text{ },\label{BS_exp_1:eq1}
\end{align}
where 
$\sigma_i^2$ is the asymptotic variance of $\hat\beta_T(u_i)$ ($i=1,2$) given in Theorem \ref{thm:1},
and $\sigma_{12}$ is the asymptotic covariance of $\hat\beta_T(u_1)$ and $\hat\beta_T(u_2)$. 
-see Sec. \ref{prrof of joint} for more. 
Note that when the sample size $T$ diverges to infinity ($h\to 0$), $\hat\beta_T(u_1)$ and $\hat\beta_T(u_2)$ ($u_1>u_2$) are, in fact, asymptotically uncorrelated, but when the sample size is moderate compared to the difference $u_1-u_2$, they are correlated.

The asymptotic distribution \eqref{BS_exp_1:eq1} leads the test statistic
$\hat{M}_{T,B} := Th|\hat\beta_T(u_1) - \hat\beta_T(u_2)|^2/(\hat{\xi}_{T,B}^{*})^2$, 
where 
\begin{align}
(\hat{\xi}_{T,B}^{*})^2
&= \frac{Th}{B}\sum^{B}_{k=1}\bigg[\left(\hat\beta_T^{*(k)}(u_1) - \hat\beta_T(u_1)\right)^2
+\left(\hat\beta_T^{*(k)}(u_2) - \hat\beta_T(u_2)\right)^2\notag\\
&\quad\quad\quad\quad\quad -2\left(\hat\beta_T^{*(k)}(u_1) - \hat\beta_T(u_1)\right)
\left(\hat\beta_T^{*(k)}(u_2) - \hat\beta_T(u_2)\right)\bigg],\notag
\end{align}
and the notation $\hat\beta_T^{*(k)}(\cdot)$ denotes the multiplier bootstrap estimator obtained by \eqref{BS_local_SLADE} and $B$ is the numbers of bootstrap replications. 
Under $H_{0}$, it is shown that $\hat{M}_{T,B} \dlim \chi^2_1$. 
Therefore, we reject the null hypothesis if $\hat{M}_{T,M}$ exceeds a quantile of $\chi^2_1$ distribution corresponding to the user-specified significance level.
On the other hand, we can show that $\hat{M}_{T,B}$ diverges to infinity in probability under the alternative. 
Thus, the proposed test is consistent.

We check the finite sample performance of the test based on $\hat{M}_{T,B}$ under the following setting. Consider the point $u_1=0.2$ so that ($\beta(u_1) \approx 0.4702$), and $u_2$ is chosen as $u_2 = 0.7, 0.75$ and $0.8$.  Then $\beta_{1}(u_2)$ is (approximately) equal to
0.4702, 0.0000 and -0.4702.
The case $u_2 = 0.7$ corresponds to the null hypothesis, while the cases 
$u_2 = 0.75$ and $0.8$ correspond to the alternatives.
%The sample size is $T=100, 500$ and $1000$, and the number of bootstrap is $M=1000$.
%The error distribution is one of $N(0,1)$, $t_2$ and Cauchy.
We use the estimator  labeled as LSW2q2 in the previous section. 
For the significance level $\delta = 0.10$ or $0.05$, the empirical rejection rates based on 1000 replications are shown in  Table \ref{tbl:BS2} using $B=1000$.
{The empirical sizes of the proposed test under the null hypothesis are in good agreement with the nominal levels for all distributions considered.} %except for the symmetric $\alpha$-stable distribution, where larger sample sizes are needed.} 
{Furthermore, the power of the test (see the columns for $u_2=0.75$ and $u_2=0.80$) increases rapidly across all scenarios.}

\begin{table}[htbp]
\centering
\caption{Empirical rejection rate of $\hat{M}_{T,B}$-test  statistic using the LSW2q2 estimator and $B=1000$ bootstrap replications. }\label{tbl:BS2}
\medskip
%\begin{tabular}{ll}
%\begin{tabular}{llccc}
%	\multicolumn{5}{l}{$\epsilon_t\sim N(0,1)$, $\delta = 0.100$}\\
%	$u_2$ &			& 0.70 & 0.75 & 0.80 \\\hline
%	$T$ 	& 100 & 0.065 & 0.072 & 0.099 \\
%	& 500 & 0.082 & 0.312 & 0.779 \\
%	& 1000 & 0.074 & 0.556 & 0.977 \\\hline
%\end{tabular}
%&
%\begin{tabular}{llccc}
%	\multicolumn{5}{l}{$\epsilon_t\sim N(0,1)$, $\delta = 0.050$}\\
%	$u_2$ &			& 0.70 & 0.75 & 0.80 \\\hline
%	$T$ 	& 100 & 0.041 & 0.045 & 0.058 \\
%	& 500 & 0.036 & 0.216 & 0.677 \\
%	& 1000 & 0.035 & 0.427 & 0.945 \\\hline
%\end{tabular}
%\\\\
%\begin{tabular}{llccc}
%	\multicolumn{5}{l}{$\epsilon_t\sim t_2$, $\delta = 0.100$}\\
%	$u_2$ &			& 0.70 & 0.75 & 0.80 \\\hline
%	$T$ 	& 100 & 0.085 & 0.096 & 0.118 \\
%	& 500 & 0.103 & 0.309 & 0.763 \\
%	& 1000 & 0.096 & 0.596 & 0.973 \\\hline
%\end{tabular}
%&
%\begin{tabular}{llccc}
%	\multicolumn{5}{l}{$\epsilon_t\sim t_2$, $\delta = 0.050$}\\
%	$u_2$ &			& 0.70 & 0.75 & 0.80 \\\hline
%	$T$ 	& 100 & 0.045 & 0.054 & 0.069 \\
%	& 500 & 0.053 & 0.218 & 0.657 \\
%	& 1000 & 0.042 & 0.449 & 0.945 \\\hline
%\end{tabular}
%\\\\
%\begin{tabular}{llccc}
%	\multicolumn{5}{l}{$\epsilon_t\sim$ Cauchy, $\delta = 0.100$}\\
%	$u_2$ &			& 0.70 & 0.75 & 0.80 \\\hline
%	$T$ 	& 100 & 0.108 & 0.116 & 0.127 \\
%	& 500 & 0.109 & 0.384 & 0.856 \\
%	& 1000 & 0.087 & 0.703 & 0.996 \\\hline
%\end{tabular}
%&
%\begin{tabular}{llccc}
%	\multicolumn{5}{l}{$\epsilon_t\sim$ Cauchy, $\delta = 0.050$}\\
%	$u_2$ &			& 0.70 & 0.75 & 0.80 \\\hline
%	$T$ 	& 100 & 0.071 & 0.071 & 0.088 \\
%	& 500 & 0.063 & 0.273 & 0.778 \\
%	& 1000 & 0.047 & 0.598 & 0.986 \\\hline
%\end{tabular}
%\end{tabular}
%\end{table}

\begin{tabular}{ccccc}\hline
	\multicolumn{5}{l}{$\epsilon_t\sim N(0,1)$, $\delta=0.100$}\\
	$u_2$ &     & 0.70 & 0.75 & 0.80 \\\hline
	$T$ & 100   & 0.082 & 0.096 & 0.151 \\
	& 500       & 0.093 & 0.548 & 0.976 \\
	& 1000      & 0.089 & 0.826 & 1.000 \\\hline
\end{tabular}
%%% Table 1-2
\begin{tabular}{ccccc}\hline
	\multicolumn{5}{l}{$\epsilon_t\sim N(0,1)$, $\delta=0.050$}\\
	$u_2$ &  & 0.70 & 0.75 & 0.80 \\\hline
	$T$ & 100   & 0.038 & 0.050 & 0.096 \\
	& 500       & 0.046 & 0.425 & 0.951 \\
	& 1000      & 0.041 & 0.734 & 0.999 \\\hline
\end{tabular}\\
\bigskip

%%% Table 2-1
\begin{tabular}{ccccc}\hline
	\multicolumn{5}{l}{$\epsilon_t\sim t_2$, $\delta=0.100$}\\
	$u_2$ &  & 0.70 & 0.75 & 0.80 \\\hline
	$T$ & 100   & 0.106 & 0.129 & 0.216 \\
	& 500       & 0.100 & 0.662 & 0.995 \\
	& 1000      & 0.105 & 0.925 & 1.000 \\\hline
\end{tabular}
%%% Table 2-2
\begin{tabular}{ccccc}\hline
	\multicolumn{5}{l}{$\epsilon_t\sim t_2$, $\delta=0.050$}\\
	$u_2$ &  & 0.70 & 0.75 & 0.80 \\\hline
	$T$ & 100   & 0.061 & 0.078 & 0.134 \\
	& 500       & 0.051 & 0.552 & 0.989 \\
	& 1000      & 0.063 & 0.863 & 1.000 \\\hline
\end{tabular}\\
\bigskip

%%% Table add_et5-1
\begin{tabular}{ccccc}\hline
	\multicolumn{5}{l}{$\epsilon_t\sim t_{1.5}$, $\delta=0.100$}\\
	$u_2$ &  & 0.70 & 0.75 & 0.80 \\\hline
	$T$ & 100   & 0.124 & 0.158 & 0.231\\
	& 500	& 0.116 & 0.693 & 0.999\\
	& 1000  & 0.107 & 0.935 & 1.000\\\hline
\end{tabular}
%%% Table add_et5-2
\begin{tabular}{ccccc}\hline
	\multicolumn{5}{l}{$\epsilon_t\sim t_{1.5}$, $\delta=0.050$}\\
	$u_2$ &  & 0.70 & 0.75 & 0.80 \\\hline
	$T$ & 100  & 0.059 & 0.093 & 0.153\\
	& 500  & 0.059 & 0.571 & 0.999\\
	& 1000 & 0.052 & 0.909 & 1.000\\\hline
	
\end{tabular}\\
\bigskip

%%% Table 3-1
\begin{tabular}{ccccc}\hline
	\multicolumn{5}{l}{$\epsilon_t\sim \text{Cauchy}$, $\delta=0.100$}\\
	$u_2$ &     & 0.70 & 0.75 & 0.80 \\\hline
	$T$ & 100   & 0.179 & 0.203 & 0.264 \\
	& 500       & 0.110 & 0.685 & 0.999 \\
	& 1000      & 0.116 & 0.959 & 1.000 \\\hline
\end{tabular}
%%% Table 3-2
\begin{tabular}{ccccc}\hline
	\multicolumn{5}{l}{$\epsilon_t\sim \text{Cauchy}$, $\delta=0.050$}\\
	$u_2$ &  & 0.70 & 0.75 & 0.80 \\\hline
	$T$ & 100   & 0.121 & 0.140 & 0.195 \\
	& 500       & 0.064 & 0.581 & 0.999 \\
	& 1000      & 0.065 & 0.927 & 1.000 \\\hline
\end{tabular}\\
\bigskip

%%%% Table 4-1
%\begin{tabular}{ccccc}\hline
%	\multicolumn{5}{l}{$\epsilon_t\sim \text{S$\alpha$S}$, $\delta=0.100$}\\
%	$u_2$ &  & 0.70 & 0.75 & 0.80 \\\hline
%%	$T$ & 100   & 0.341 & 0.356 & 0.412 \\
%$T$	& 500       & 0.163 & 0.482 & 0.955 \\
%	& 1000      & 0.174 & 0.774 & 0.994 \\\hline
%\end{tabular}
%%%% Table 4-2
%\begin{tabular}{ccccc}\hline
%	\multicolumn{5}{l}{$\epsilon_t\sim \text{S$\alpha$S}$, $\delta=0.050$}\\
%	$u_2$ &  & 0.70 & 0.75 & 0.80 \\\hline
%%	$T$ & 100   & 0.279 & 0.302 & 0.341 \\
%$T$	& 500       & 0.107 & 0.389 & 0.944 \\
%	& 1000      & 0.112 & 0.712 & 0.994 \\\hline
%\end{tabular}\\
\bigskip
\end{table}

%\begin{remark}
%We can also construct a test statistic
%\[
%\tilde{M}_{T,M} := \sqrt{Th}(\hat\beta_T(u_1) - \hat\beta_T(u_2))/\hat{\xi}_{T,M}^*,
%\] 	
%which converges to $N(0,1)$. 
%On the other hand, the statistic $\hat{M}_{T,M}$ is easily extended to the multivariate parameter case ($p\geq2$).
%\end{remark}

%
\vspace{0.5cm}

\section{Real data analysis}\label{sec:Data analysis}

For illustration of  methodology, we examine the daily squared log-returns of the DAX index over the period from mid-1991 to mid-1998. The data are taken from the built-in \texttt{R} dataset \texttt{EuStockMarkets}. The dataset contains 1,860 observations, where each observation corresponds to a trading day (approximately 260 trading days per year).
Figure \ref{fig:exploratoryplot} presents some basic exploratory plots of the data {after centering them.} The upper panel illustrates a well-known stylized fact of financial time series, namely that squared log-returns exhibit serial dependence (volatility clustering). The lower panel provides evidence that the data are heavy-tailed; in particular, Fig. \ref{fig:exploratoryplot}(d) shows an upward trend in the mean excess plot, which is characteristic of heavy-tailed distributions (see \citet[Ch. 6.2.2]{Embrechtsetal(1997)}). {The tail exponent $\alpha$ is estimated to about 1.54 conforming further that these data are heavy-tailed}.

\begin{figure}[t]
	\centering
	\includegraphics[width=0.7\linewidth]{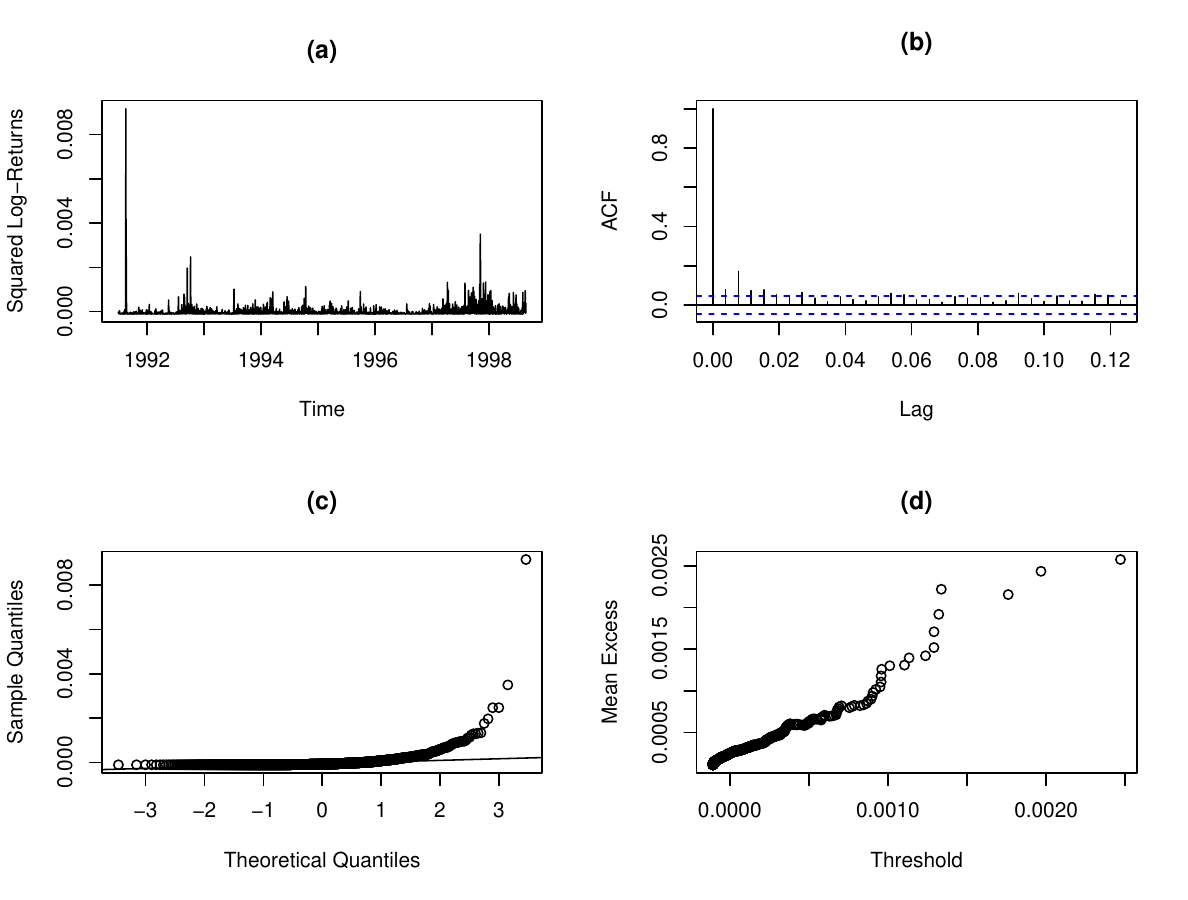}
	\caption{(a) Time series plot of {centered} squared log-returns of the DAX index.
		(b) Autocorrelation function (ACF) of the squared log-returns.
		(c) Quantile–quantile (Q–Q) plot of the squared log-returns.
		(d) Mean excess plot over increasing thresholds.}
\label{fig:exploratoryplot}
\end{figure}

Denote by $Y_{t,T}$ the {centered} squared log-returns. We fit a sequence of time-varying autoregressive models, tv-AR($p$), as defined in \eqref{ar}, for orders $p = 1, 2, 3, 4$. Estimation is performed using the self-weighted estimator LSW2q2, employing the same bandwidth parameter $h$ as in the previous section, as well as several additional values. Note that the choice $h = 0.08$ corresponds to $h = \log(T) / T^{3/5}$.

To select the model order, we consider, for each $u \in (0,1]$, a localized version of the Bayesian Information Criterion (BIC), following \citet{Xuetql(2022)}, and then average these values over time. More precisely, for a fixed $u$, we define
\[
\mathrm{BIC}(u,p) = \log \Biggl( \sum_{t = p+1}^{T}
K_h\bigl(t - \lfloor uT \rfloor\bigr)\,
w_{t-1,T} \, \bigl| e_{t-1,T}(\beta) \bigr| \Biggr)
+ p \log(T),
\]
where we recall the definitions in \eqref{ar} and \eqref{local_SLADE}. We then average over the time grid $\{i/T : i = 1, \dots, T\}$ and select the order $p$ that minimizes
\begin{align}
	\mathrm{BIC}(p)
	= \frac{1}{T} \sum_{i=1}^{T} \mathrm{BIC}\!\left(\frac{i}{T}, p\right).
	\label{eq:BIC}
\end{align}

This provides a heuristic approach to model selection in the present framework but requires further theoretical investigation. For the dataset under consideration, Table~\ref{tableBIC} shows that, across a range of bandwidth values, the BIC consistently favors a tv-AR(1) specification. Figure~\ref{fig:estimationandseplot} displays the resulting estimator along with 95\% point-wise confidence intervals; see Theorem \ref{thm:BS}. 
{The estimated curve appears to represent a multi modal nonlinear function three local maxima and a local minimum with a larger uncertainty at the boundaries. Obviously, the true regression function is substantially more complex than a linear or low-order polynomial trend. A possible explanation of the largest peak could be  the interaction of Bundesbank rate cuts, exchange-rate movements, and optimism regarding the future euro currency.}

{
To further investigate whether the parameter estimate remains stable over time, we employ the testing methodology described in Sec.~\ref{ssec:6.1}. Figure~\ref{fig:estimationandseplot} suggests the presence of local maxima at approximately $u=0.10$, $u=0.27$, and $u=0.70$ (the global maximum), as well as local minima at approximately $u=0.20$ and $u=0.37$ (the global minimum). We first test the hypothesis $ H_1: {\beta}(0.70)={\beta}(0.37), $ since this comparison involves the global maximum and global minimum of the estimated curve and therefore corresponds to the largest  difference of the plot. In addition, we consider the hypotheses  $H_2: {\beta}(0.10)={\beta}(0.27) $ and $H_3: {\beta}(0.27)={\beta}(0.37), $ which compare neighboring peaks and troughs and are thus sensitive to the oscillatory behavior exhibited in the first half of the domain. A direct application of the methodology described in Sec.~\ref{ssec:6.1}, using $B=5000$ bootstrap replications, leads to the rejection of all three hypotheses, with $p$ -values that are effectively equal to zero.}

\begin{table}[H]
	\centering
	\caption{BIC values, computed via \eqref{eq:BIC}, for the model specified in \eqref{ar}, as a function of the model order and the bandwidth parameter. The models are fitted by employing the estimator LSW2q2.}
	\label{tableBIC}
	\medskip
\begin{tabular}{|c|cccc|} \hline
Bandwidth  & $p=1$ & $p=2$  & $p=3$  & $p=4$  \\	\hline
$h=0.08$	   &  13544.39 & 13586.80 &  13592.32 &  13598.56 \\ 	\hline
$h=0.12$	   &  13570.41 & 13607.28 &  13612.61 &  13618.05  \\  \hline
$h=0.16$   &  13583.61 & 13619.55 &  13622.58 &  13628.98  \\  \hline
$h=0.20$   &  13597.19 & 13627.22 &  13628.14 &  13.633.37 \\ \hline
\end{tabular}
\end{table}

\begin{figure}[H]
	\centering
	\includegraphics[width=0.7\linewidth]{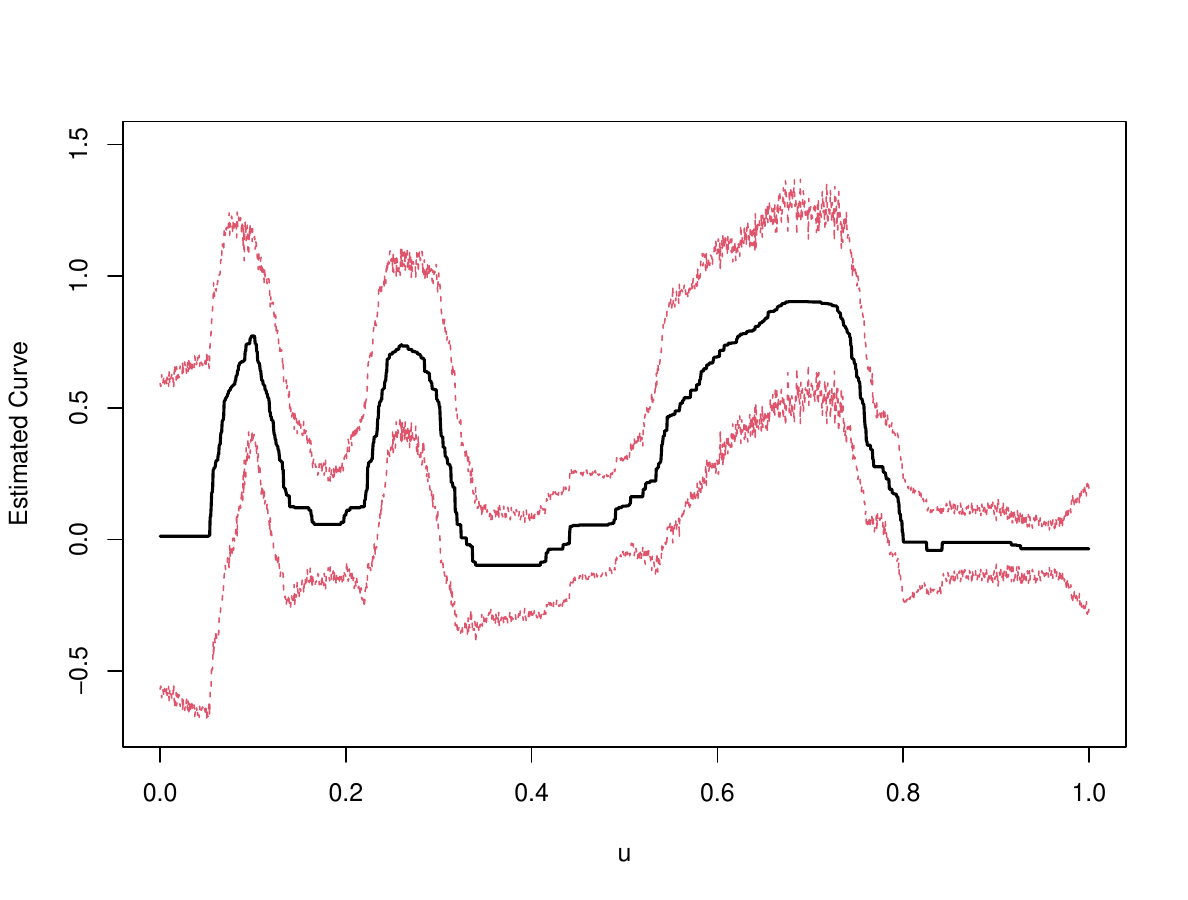}
	\caption{Estimated curve (black line)  $\hat{\beta}(.)$ and 95\% pointwise confidence intervals (red line) after fitting a tv-AR(1) model to squared log-returns data of the DAX index.}
	\label{fig:estimationandseplot}
\end{figure}

\section*{Acknowledgments}
We are sincerely grateful to the Editor, Associate Editor, and the two reviewers for their thoughtful and constructive comments on an earlier version of the manuscript. Their feedback has significantly strengthened the paper and led to substantial improvements.
The first author was supported by  JSPS KAKENHI Grant Numbers JP20K13467 and JP24K04815.
The third author was supported by  JSPS KAKENHI Grant-in-Aid for Scientific Research (C) 23k11004 and Nanzan University Pache Research Subsidy I-A-2 for the 2026 academic year (Junichi Hirukawa).
%}
%%%%} \fi

%\section*{Data Availability Statement}
%The data that support the findings of this study are openly available at \href{https://www.nasdaq.com/market-activity/quotes/historical}{https://www.nasdaq.com/market-activity/quotes/historical}.

%\bibliographystyle{chicagoa}
%\bibliography{RobustBIB.bib}

%$\begin{center}
%{\LARGE Supplement to ``Inference for Non-Stationary Heavy Tailed Time Series''}
%\end{center}

\renewcommand{\theequation}{A-\arabic{equation}}
\renewcommand{\thelemma}{A-\arabic{lemma}}
\setcounter{equation}{0}
\setcounter{lemma}{0}
\renewcommand{\thesubsection}{A-\arabic{subsection}}
\setcounter{subsection}{0}
\section*{Appendix} %\pagenumbering{arabic}

This section gives some auxiliary lemmas and proofs of main results.
Throughout this section, denote $\zeta_T := \sqrt{Th}$.  In addition, $C$ denotes a generic positive constant, independent of $T$ and maybe different depending on context.

\subsection{An auxiliary lemma}

Define a non-negative  stochastic process
\begin{align}
A_t := \sum^{\infty}_{j=0}l_j^{-1}|\epsilon_{t-j}|,\notag
\end{align}
where $\{l_j\}$ is an identical  sequence of coefficients as introduced Assumption \ref{ass:2}. Then,  $\{A_t\}$ exists and is finite almost surely.
\begin{lemma}\label{lem:At}
Suppose that Assumptions \ref{ass:1}, \ref{ass:2} and \ref{ass:n4} hold.
Then
$\sum^{n}_{t=1}A_t=O_p(a_n+ nb_n)$ and $\sum^{n}_{t=p+1}A_t^2=O_p(a_n^2)$.
\end{lemma}

\begin{proof}
We check the conditions of \citet[Thms 4.1 and 4.2]{davis1985limit}.
By Assumption \ref{ass:2}, the coefficients $\{l_j^{-1}\}$ satisfy the summability condition (2.6) of \cite{davis1985limit}.
Moreover, under Assumption \ref{ass:1}, $Z_t := |\epsilon_t|$ satisfies the conditions (1.1) and (1.2) of \cite{davis1985limit}. %\erase{with $p=1$ and $q=0$}.
Then, $a_n^{-1}(\sum^{n}_{t=1} A_t - n b_n(\sum^{\infty}_{j=0}l_j^{-1}))$ converges to a stable random variable when $n\to \infty$.
Therefore, we have
\begin{align}
\sum^{n}_{t=p+1}A_t
&= a_n\left(\frac{\sum^{n}_{t=1} A_t - n b_n\left(\sum^{\infty}_{j=0}l_j^{-1}\right)}{a_n}\right)
+ \left(\sum^{\infty}_{j=0}l_j^{-1}\right)nb_n.\notag
\end{align}
Then we get the desired result. We can also show the result for $\sum^{n}_{t=1}A_t^2$ by Theorem 4.2 of \cite{davis1985limit}.
\end{proof}

\subsection{Convexity  lemma}

{
The lemma was originally established by \citet{AndersenandGill(1982)}. Subsequently, \citet{Pollard(1991)} presented a more direct proof based on an elegant convexity argument, whereby the convexity of the criterion function plays a central role in establishing uniform approximation.}

\begin{lemma}
\label{Lemma:Convexity}
{	Let $\left\{\lambda_T\left(\theta\right)\mid \theta \in \Theta\right\}$ be a sequence of random convex functions defined on a convex, open subset $\Theta$ of $\mathbb{R}^{p}$. Suppose $\lambda\left(\cdot\right)$ is a real-valued function on $\Theta$ for which $\lambda_T\left(\theta\right)\to\lambda\left(\theta\right)$ in probability, for each $\theta$ in $\Theta$. Then for each compact subset $K$ of $\Theta$,
	\begin{align*}
		\sup_{\theta\in K}\left|\lambda_T\left(\nu\right)-\lambda\left(\nu\right)\right|\to0\qquad \mbox{in probability.}
	\end{align*}
	The function $\lambda\left(\cdot\right)$ is necessarily convex on $\Theta$. }
\end{lemma}

\subsection{Proofs of Propositions \ref{Prop: stationary approx} and \ref{prop:2}}

\begin{proof}[Proof of Proposition \ref{Prop: stationary approx}:]

\begin{enumerate}
\item By arguing as in \citet[Prop. 13.3.1-13.3.2]{BrockwellandDavis(1991)} and employing Assumption \ref{ass:1} we have that the process
defined by \eqref{def: Def nonstationary process} is well-defined. The process defined by \eqref{def: Def stationary process} is also
well-defined, by using Assumption \ref{ass:2}. In addition, it is strictly stationary and ergodic.
We study the difference  $| Y_{t,T}- Y_{t}(u) |$ for $u \in (0,1)$ such that $| t/T-u | < T^{-1}$. We have that
\begin{align*}
	| Y_{t,T}- Y_{t}(u) | \leq | Y_{t,T}- Y_{t}(t/T) | + |Y_{t}(t/T) - Y_{t}(u) |.
\end{align*}
But
\begin{align*}
	| Y_{t,T}- Y_{t}(t/T) |     & \leq \sum_{j=0}^{\infty} |\psi_{t,T}(j)-\psi(t/T, j)| |\epsilon_{t-j}| 
	\leq \frac{\blue{C_0}}{T} \sum_{j=0}^{\infty}  \frac{1}{l_j} |\epsilon_{t-j}|,
\end{align*}
by using Assumption \ref{ass:2} and noting that right-hand side is well defined and $O_{p}(1)$.
Furthermore,
\begin{align*}
	|Y_{t}(t/T) - Y_{t}(u) |  & \leq \sum_{j=0}^{\infty}  | \psi(t/T, j)- \psi(u,j)| |\epsilon_{t-j}| 
	\leq \blue{C_0} \left|\frac{t}{T}-u\right| \sum_{j=0}^{\infty} \frac{1}{l_j} |\epsilon_{t-j}|,
\end{align*}
by using again Assumption \ref{ass:2}. Collecting the above inequalities we obtain the desired result.
\item  By considering a Taylor expansion of the function $h(x) := x g(x)$ and taking into account  Assumption \ref{ass:3}, we have
that
\begin{align*}
	\|w_{t-1,T}X_{t-1,T} - w_{t-1}(t/T) X_{t-1}(t/T)\|
	&\leq \sup_{x\in\R^p}\|h'(x)\|\|X_{t-1,T}-X_{t-1}(t/T)\|\notag\\
	&= {M}\|X_{t-1,T}-X_{t-1}(t/T)\|.
\end{align*}
Now, the triangle inequality and the fact that for any $x_1,...,x_p$,  $\sum^{p}_{i=1}x_i^2\leq \left(\sum^{p}_{i=1}|x_i|\right)^2$, show that
\begin{align*}
	\| X_{t-1,T}- X_{t-1}(u) \| & \leq \|X_{t-1,T}-X_{t-1}(t/T)\| + \|X_{t-1}(t/T)-X_{t-1}(u)\| \\
	& \leq  \sum^{p}_{i=1}\left|Y_{t-i, T}-Y_{t-i}(t/T)\right|
	+ \sum^{p}_{i=1}\left| Y_{t-i}(t/T)-Y_{t-i}(u)\right|,
\end{align*}
and the result follows by the first part of the proposition.

\end{enumerate}

\end{proof}

\begin{proof}[Proof of Proposition \ref{prop:2}]
It is sufficient to follow the steps for proving \citet[Prop. 2.4]{DahlhausandPolonik(2009)}. For instance, \citet[Prop. 2.4]{DahlhausandPolonik(2009)}
show that $|\psi_{j, t, T}| \leq K \rho^{j}$ for some positive constant and $ \rho < 1$. Clearly the choice of such weights satisfies assumption \ref{ass:2}(i), in the case
of tvAR($p$) models.  We can verify the rest of conditions of  Assumption \ref{ass:2} in a similar way.		
\end{proof}

\subsection{Proof of Theorem \ref{thm:1}}

Before proving Theorem \ref{thm:1} we give two useful lemmas.  Define
\begin{align}
\Tdm_T = \{\max(p+1,{\intp[u_0T]}-\lfloor C_KTh\rfloor),...,\min(T,{\intp[u_0T]}+\lfloor C_KTh\rfloor)\}.\notag
\end{align}
\blue{Then}, the quantity $K\left((t-{\intp[u_0T]})/(Th)\right)$ is zero if $t\notin\Tdm_T$.
By the inequality $u_0T-1< \intp[u_0T]\leq u_0T$, we have 
$$
\Tdm_T = \{{\intp[u_0T]} - \lfloor C_K Th\rfloor,..., {\intp[u_0T]} + \lfloor C_K Th\rfloor\}
$$
for sufficient large $T$.

\begin{lemma}
\label{lem:3}
Under the conditions of  Theorem \ref{thm:1}, the quantity
\[
W_T := -\frac{1}{\sqrt{Th}}\sum_{{t\in\Tdm_T}}w_{t-1}(u_0)K\left(\frac{t-{\intp[u_0T]}}{Th}\right)\mathrm{sign}(\epsilon_t)X_{t-1}(u_0)
\]
converges to $N(0_p,\kappa_2 \Omega(u_0))$ in distribution, as $T\to\infty$.
\end{lemma}
\begin{proof}
{ Define $g_{t,T}:=\zeta_T^{-1}w_{t-1}(u_0)K((t-{\intp[u_0T]})/(Th))\mathrm{sign}(\epsilon_t)X_{t-1}(u_0)$.
Then, we have $W_T = -\sum_{t\in\Tdm_T}g_{t,T}$. 
%Noting that the median of $\epsilon_t$ is zero, we have
%\begin{align*}
%	\E[g_{t,T}]
%	= \E[\E(g_{t,T}\mid \Filt_{t-1,T})]
%	= \frac{1}{\zeta_T}K\left(\frac{t-{\intp[u_0T]}}{Th}\right)E[w_{t-1}(u_0)X_{t-1}(u_0)\E( \mathrm{sign}(\epsilon_t) \mid \Filt_{t-1,T})]
%	= 0_p.
%\end{align*}
To show the central limit theorem for $W_T$, we check the condition of 
\citet[Thoerem 2.3]{McLeish1974dependent}.
For any constant vector $c\in\R^p$, \blue{consider} an array 
$\{c^\top g_{t,T}:t\in\Tdm_T\}$ \blue{and a sequence of filtrations
\[
\Filt_{t-1,T}:=\sigma\{\epsilon_s: s\in\Tdm_T, s\leq t-1\}\quad (t\in\Tdm_T, T\in\N).
\]
By noting that the median of $\epsilon_t$ is zero, we have
\[
\E(c^\top g_{t,T}|\Filt_{t-1,T})=
\zeta_T^{-1}w_{t-1}(u_0)K((t-{\intp[u_0T]})/(Th))c^\top X_{t-1}(u_0)\E(\mathrm{sign}(\epsilon_t)) = 0.
\]
This implies that 
}
$\{c^\top g_{t,T}:t\in\Tdm_T\}$
is a sequence of martingale difference with respect to $\{\Filt_{t,T}:t\in\Tdm_T\}$.
We bound $\max_{t\in\Tdm_T}|c^\top g_{t,T}|$ as
\begin{align*}
	\max_{t\in\Tdm_T}|c^\top g_{t,T}| 
	&\leq \|c\|\max_{t\in\Tdm_T}\|g_{t,T}\|\notag\\
	&= \frac{\|c\|}{\zeta_T}\max_{t\in\Tdm_T}\|w_{t-1}(u_0)X_{t-1}(u_0)\|\notag\\
	&\leq \frac{\|c\|M}{\zeta_T}\quad (\text{by Assumption \ref{ass:3}})\notag\\
	&\to 0\quad (T\to\infty).
\end{align*}
This implies that the conditions (a) and (b) of 
\citet[Theorem 2.3]{McLeish1974dependent} are fulfilled. But 
\begin{align}
	\sum_{t\in\Tdm_T}\left\{c^\top g_{t,T}\right\}^2
	&=  c^\top \left\{\sum_{t\in\Tdm_T}g_{t,T}g_{t,T}^\top\right\}c\notag\\
	&= 
	\frac{1}{Th}\sum_{t\in\Tdm_T}K^2\left(\frac{t-{\intp[u_0T]}}{Th}\right)w^2_{t-1}(u_0) c^\top X_{t-1}(u_0)X_{t-1} ^\top(u_0) c.\label{eq:condition3}
\end{align}
The process $\{w_{t-1}^2(u_0) c^\top X_{t-1}(u_0)X_{t-1} ^\top(u_0) c: t\in\Z\}$ is an ergodic sequence with 
finite $L_1$-norm. 
Then, by slightly modification of \citet[Lemma A.1]{DahhausandRao(2006)}, the right hand side of \eqref{eq:condition3} converges to 
\begin{align*}
	&\frac{1}{Th}\sum_{t\in\Tdm_T}K^2\left(\frac{t-{\intp[u_0T]}}{Th}\right)
	w_{t-1}^2(u_0)c^\top X_{t-1}(u_0)X_{t-1}^\top (u_0) c\notag\\
	&\xrightarrow{\mathrm{a.s.}} \int^{C_K}_{-C_K}K(v)^2dv
	\left[c^\top\E( w_{t-1}^2(u_0)  X_{t-1}(u_0)X_{t-1}^\top(u_0) )c\right]\notag\\
	&=\kappa_2c^\top \Omega(u_0) c.
\end{align*}
\blue{This implies 
	$\sum_{t\in\Tdm_T}\left\{c^\top g_{t,T}\right\}^2/(\kappa_2c^\top \Omega(u_0) c) \xrightarrow{\mathrm{a.s.}}1$, 
	so condition (c) of \citet[Theorem 2.3]{McLeish1974dependent} holds true.} 
Hence
$c^\top W_T = \sum_{t\in\Tdm_T}c^\top g_{t,T}\dlim N(0, \kappa_2c^\top \Omega(u_0) c)$.
Since $c$ is chosen arbitrary, we obtain the desired result by the Cram\'er-Wold device.}
\end{proof}

\begin{lemma}
\label{lem:4}
Under the conditions in Theorem \ref{thm:1}, the quantity
\[
S_T := \sum_{{t\in \Tdm_T}}w_{t-1}(u_0)K\left(\frac{t-{\intp[u_0T]}}{Th}\right)\tilde \delta_{t,T}(u_0,v)
\]	
converges to $(f(0)/2) v^\top \Sigma(u_0)v$ in probability as $T\to\infty$,
where
\[
\tilde\delta_{t,T}(u_0,v) := \int^{\zeta_T^{-1}v^\top X_{t-1}(u_0)}_{0}\{\mathbb{I}(\epsilon_t\leq s)
- \mathbb{I}(\epsilon_t\leq 0)\}ds,
\]
and $v \in \mathbb{R}^{p}$.
\end{lemma}

\begin{proof}
{ Let us consider the decomposition $S_T=S^{(1)}_T + S^{(2)}_T$, where
\begin{align}
	&S^{(1)}_T = \sum_{t\in\Tdm_T}w_{t-1}(u_0)K\left(\frac{t-{\intp[u_0T]}}{Th}\right) \E[\tilde\delta_{t,T}(u_0,v)\mid \Filt_{t-1,T}],\notag\\
	&S^{(2)}_T = \sum_{t\in\Tdm_T}w_{t-1}(u_0)K\left(\frac{t-{\intp[u_0T]}}{Th}\right)\left(\tilde\delta_{t,T}(u_0,v) - \E[\tilde\delta_{t,T}(u_0,v)\mid \Filt_{t-1,T}]\right).\notag
\end{align}
We evaluate $S^{(1)}_T$ and $S^{(2)}_T$ in succession.
\medskip
First, $S^{(1)}_T$ is calculated as 
\begin{align}
	S^{(1)}_T
	& = \sum_{t\in\Tdm_T}w_{t-1}(u_0)K\left(\frac{t-{\intp[u_0T]}}{Th}\right) 
	\int^{\zeta_T^{-1}v^\top X_{t-1}(u_0)}_{0}\{F(s) - F(0)\}ds
	\notag\\
	& = \sum_{t\in\Tdm_T}w_{t-1}(u_0)K\left(\frac{t-{\intp[u_0T]}}{Th}\right) 
	\int^{\zeta_T^{-1}v^\top X_{t-1}(u_0)}_{0}\left\{s f(0) + \frac{1}{2}s^2 f'(s^*)\right\}ds
	\notag\\
	& = \frac{f(0)}{2Th}\sum_{t\in\Tdm_T}K\left(\frac{t-{\intp[u_0T]}}{Th}\right) v^\top w_{t-1}(u_0)X_{t-1}(u_0) X_{t-1}^\top(u_0) v\label{eq:p22}\\
	& +\frac{1}{2}\sum_{t\in\Tdm_T}w_{t-1}(u_0)K\left(\frac{t-{\intp[u_0T]}}{Th}\right)\int^{\zeta_T^{-1}v^\top X_{t-1}(u_0)}_{0}s^2 f'(s^*)ds,\label{eq:p23}
\end{align}
where $s^*$ lies on the segment joining $0$ and $s$}.
Since each element of $\{w_{t-1}(u_0)X_{t-1}(u_0)X_{t-1}^\top(u_0): t\in\Z\}$ is a strictly stationary ergodic process with finite expectation, the term \eqref{eq:p22} converges almost surely to 
\[
\frac{f(0)}{2}v^\top \E[w_{t-1}(u_0)X_{t-1}(u_0) X_{t-1}^\top(u_0)]v = \frac{f(0)}{2} v^\top \Sigma(u_0)v,
\]
using \citet[Lemma A.1]{DahhausandRao(2006)}. 
On  the other hand, the term \eqref{eq:p23} is bounded as
\begin{align}
&\left|\frac{1}{2}\sum_{t\in\Tdm_T}w_{t-1}(u_0)K\left(\frac{t-{\intp[u_0T]}}{Th}\right)\int^{\zeta_T^{-1}v^\top X_{t-1}(u_0)}_{0}s^2 f'(s^*)ds\right|\notag\\
&\leq \frac{1}{2}\sum_{t\in\Tdm_T}w_{t-1}(u_0)K\left(\frac{t-{\intp[u_0T]}}{Th}\right)
f'_{\mathrm{sup}}\frac{1}{3}\left|\zeta_T^{-1}v^\top X_{t-1}(u_0)\right|^3 \notag\\
&\leq \frac{f'_{\mathrm{sup}}}{\zeta_T^3}\|v\|\sum_{t\in\Tdm_T}K\left(\frac{t-{\intp[u_0T]}}{Th}\right)
w_{t-1}(u_0)\left|X_{t-1}(u_0)\right|^3\notag\\
&\leq \frac{f'_{\sup}}{6\zeta_T}\|v\| \frac{1}{Th}\sum_{t\in\Tdm_T}K\left(\frac{t-{\intp[u_0T]}}{Th}\right)
w_{t-1}(u_0)\left\|X_{t-1}(u_0)\right\|^3\notag\\
&\leq \frac{ C f'_{\sup}}{\zeta_T}\|v\| \frac{1}{Th}\sum_{t\in\Tdm_T}K\left(\frac{t-{\intp[u_0T]}}{Th}\right)
\to 0,\notag
\end{align}
where $f'_{\mathrm{sup}}:= \sup_{x\in\R}|f'(x)|$ and the second inequality follows by Assumption \ref{new_ass:5}(i). 
Therefore, we have
\begin{align}
S^{(1)}_T \plim \frac{f(0)}{2} v^\top \Sigma(u_0)v\quad\text{as $T\to\infty$}.\label{eq:p24}
\end{align}
\medskip
In addition,  evaluation of the second moment of $S^{(2)}_T$ shows that 
\begin{align}
\E\left[\left|S^{(2)}_T\right|^2\right]&=\E\left[\left|
\sum_{t\in\Tdm_T}w_{t-1}(u_0)K\left(\frac{t-{\intp[u_0T]}}{Th}\right)\left(\tilde\delta_{t,T}(u_0, v) - \E[\tilde\delta_{t,T}(u_0, v)\mid \Filt_{t-1,T}]\right)
\right|^2\right]\notag\\
&\leq
\sum_{t\in\Tdm_T}K^2\left(\frac{t-{\intp[u_0T]}}{Th}\right)\E\left[w^2_{t-1}(u_0)\left(\tilde\delta_{t,T}(u_0, v) - \E[\tilde\delta_{t,T}(u_0, v)\mid \Filt_{t-1,T}]\right)^2
\right]\notag\\
&\leq 2\sum_{t\in\Tdm_T}K^2\left(\frac{t-{\intp[u_0T]}}{Th}\right)\E\left[w^2_{t-1}(u_0)\tilde\delta^2_{t,T}(u_0, v)\right].\label{eq:p25}
\end{align}
due to the fact that 
$\{w_{t-1}(u_0)(\tilde\delta_{t,T}(u_0, v) - \E[\tilde\delta_{t,T}(u_0, v)\mid \Filt_{t-1,T}])\}$ is an  uncorrelated process. 
Now recalling the inequality $\E\left[\left|X - \E(X|\mathcal{G})\right|^2\right]\leq 2\E(X^2)$ for any random variable $X$ and sub-field $\mathcal{G}$, the expectation $\E[w_{t-1}(u_0)^2\tilde\delta_{t,T}(u_0, v)^2]$ is evaluated as
\begin{align}
&\E\left[w_{t-1}^2(u_0) \tilde\delta^2_{t,T}(u_0, v)\right]\notag\\
&=\E\left[w_{t-1}^2(u_0) \left(\int^{|\zeta_T^{-1}v^\top X_{t-1}(u_0)|}_{0}\{\mathbb{I}(\epsilon_t\leq s)
- \mathbb{I}(\epsilon_t\leq 0)\}ds\right)^2\right]\notag\\
&\leq\E\left[w_{t-1}^2(u_0) \left(\int^{|\zeta_T^{-1}v^\top X_{t-1}(u_0)|}_{0}1ds\right)
\left(\int^{|\zeta_T^{-1}v^\top X_{t-1}(u_0)|}_{0}\{\mathbb{I}(\epsilon_t\leq s)
- \mathbb{I}(\epsilon_t\leq 0)\}ds\right)\right]\notag\\
&=\frac{\|v\|}{\zeta_T}\E\left[w^2_{t-1}(u_0) \|X_{t-1}(u_0)\|
\left(\int^{|\zeta_T^{-1}v^\top X_{t-1}(u_0)|}_{0}\{\mathbb{I}(\epsilon_t\leq s)
- \mathbb{I}(\epsilon_t\leq 0)\}ds\right)\right]\notag\\
&=\frac{\|v\|}{\zeta_T}\E\left[w_{t-1}^2(u_0) \|X_{t-1}(u_0)\|
\left(\int^{|\zeta_T^{-1}v^\top X_{t-1}(u_0)|}_{0}\{F(s)
- F(0)\}ds\right)\right]\notag\\
&=\frac{\|v\|}{\zeta_T}\E\left[w_{t-1}^2(u_0) \|X_{t-1}(u_0)\|
\left(\int^{|\zeta_T^{-1}v^\top X_{t-1}(u_0)|}_{0}sf(s^*)ds\right)\right]\notag\\
%&\qquad (\text{$s^*$ lies on the line joining $0$ and $s$})\notag\\
&\leq\frac{\|v\|^3 f_{\sup}}{2\zeta_T}\E\left[w_{t-1}^2(u_0) \|X_{t-1}(u_0)\|
|\zeta_T^{-1}v^\top X_{t-1}(u_0)|^2\right]
%\quad (f_{\sup} := \sup_{x\in\R} f(x))
\notag\\
&\leq \frac{\|v\|^3f_{\sup}}{\zeta_T^3} \E\left[w_{t-1}^2(u_0)\|X_{t-1}(u_0)\|^3\right]\notag\\
&\leq \frac{\|v\|^3f_{\sup} C}{(Th)^{3/2}}, \notag
\end{align}
using  Assumption \ref{ass:3}, where $f_{\mathrm{sup}}:= \sup_{x\in\R}f(x)$.
Therefore, \eqref{eq:p25} becomes
\begin{align}
\E\left[\left|S^{(2)}_T\right|^2\right]
&\leq {\frac{2\|v\|^3f_{\sup} C}{(Th)^{3/2}}Th\to 0 \quad(T\to\infty).}
\label{eq:p26}
\end{align}
Thus,  \eqref{eq:p24} and \eqref{eq:p26} leads the desired result.
\end{proof}

\subsubsection*{Proof of Theorem \ref{thm:1}}

\begin{proof}
Fix any $v\in\R^p$.
Let us consider the function
\begin{align*}
{L_{T}}(u_0; v) := \sum_{t\in\Tdm_T}
w_{t-1,T} K\left(\frac{t-{\intp[u_0T]}}{Th}\right)
\left\{
\left|\epsilon_t - \zeta_T^{-1} {v}^\top X_{t-1,T}\right|
- \left|\epsilon_t\right|
\right\},
\end{align*}
for each $v\in\R^p$.
\begin{comment}
Then, $L_T(u_0, v,\gamma)$ attains its minimum at  
$(v,\gamma) = (\zeta_T[\hat\beta^{\mathrm{Pol}}_{t_0;T}-\beta(u_0)],\hat\gamma)$.
In fact, when $v=\zeta_T(\beta-\beta_0)$, 
the vector $\phi_{t,T}$ has the representation
\begin{align}
	\phi_{t,T}
	&= \zeta_T\left[
	\beta-\beta(u_0) + 
	\left(\frac{t}{T}-u_0\right)\gamma 
	- (\beta(\frac{t}{T})-\beta(u_0))\right]\notag\\
	&= \zeta_T\left[
	\beta+ 
	\left(\frac{t}{T}-u_0\right)\gamma 
	-\beta(\frac{t}{T})
	\right].\notag
\end{align}
With this $\phi_{t,T}$, we have
\begin{align*}
	\zeta_T^{-1}\phi_{t,T}^\top X_{t-1,T}
	&= \left[
	\beta+ 
	\left(\frac{t}{T}-u_0\right)\gamma\right]^\top X_{t-1,T}
	-\beta(t/T)^\top X_{t-1,T}\notag\\
	&= \left[
	\beta+ 
	\left(\frac{t}{T}-u_0\right)\gamma\right]^\top X_{t-1,T}
	-(Y_{t,T}-\epsilon_{t})\notag\\
	&= -e^{\mathrm{Pol}}_{t-1,T}(\beta,\gamma) +\epsilon_{t}.
\end{align*}
Then, by substituting $v=\zeta_T(\beta-\beta(u_0))$ into $L_T$, we have
\begin{align}
	L_T(u_0, \zeta_T(\beta-\beta(u_0)),\gamma)
	&=
	\sum^{T}_{t=p+1}
	w_{t-1,T} K\left(\frac{t-}{Th}\right)
	\left\{
	\left|\epsilon_t - \{ -e^{\mathrm{Pol}}_{t-1,T}(\beta,\gamma) +\epsilon_t \}\right|
	- \left|\epsilon_t\right|
	\right\}\notag\\
	&=
	\sum_{t\in\Tdm_T}
	w_{t-1,T} K\left(\frac{t-t_0}{Th}\right)
	\left|e^{\mathrm{Pol}}_{t-1,T}(\beta,\gamma)\right|
	- \sum_{t\in\Tdm_T}
	w_{t-1,T} K\left(\frac{t-t_0}{Th}\right)\left|\epsilon_t\right|\label{eq:p11}
\end{align}
and the second term of \eqref{eq:p11} is independent of $(\beta,\gamma)$.
\end{comment}
For sufficient large $T$, the minimizer of $L_{T}(u_0; \cdot)$ coincides to $\zeta_T(\hat\beta_T(u_0) - \beta(u_0))$.
Hereafter, we derive the limit distribution of $L_{T}(u_0; v)$ for fixed {$v$}.

\medskip

By using the \cite{Knight1998}'s identity, which states that for any real numbers $x\neq0$ and $y$,
\[
|x-y|-|x|=-y\mathrm{sign}(x)+2 \int_{0}^{y}\{\mathbb{I}(x \leq s)-\mathbb{I}(x \leq 0)\} ds,
\]
we obtain
\begin{align}
L_{T}(u_0; v)
&= {v^\top Q_{T} +2\sum_{t\in\Tdm_T}w_{t-1,T}K\left(\frac{t-{\intp[u_0T]}}{Th}\right)\delta_{t,T}(v)},\label{eq:p12}
\end{align}
where
\begin{align}
&Q_{T} = -\frac{1}{\zeta_T}\sum_{t\in\Tdm_T}w_{t-1,T}K\left(\frac{t-{\intp[u_0T]}}{Th}\right)\mathrm{sign}(\epsilon_t)X_{t-1,T}\text{ and }\notag\\
& \delta_{t,T}(v) = \int^{\zeta_T^{-1}v^\top X_{t-1,T}}_{0}\{\mathbb{I}(\epsilon_t\leq s)
- \mathbb{I}(\epsilon_t\leq 0)\}ds.\label{def:def_delta}
\end{align}
In what follows, we approximate ${L_{T}(u_0; v)}$ by
\begin{align}
{\tilde L_{T}(u_0; v)
	:= v^\top W_T  +2\sum_{t\in\Tdm_T}w_{t-1}(u_0)K\left(\frac{t-{\intp[u_0T]}}{Th}\right)\tilde \delta_{t,T}(u_0,v)},
	\end{align}
	where $W_T$ and $\tilde\delta_{t,T}(u_0,v)$ have been  defined in Lemmas \ref{lem:3} and 
	\ref{lem:4}, respectively.

	\medskip
	
	For {$Q_{T}$} and $W_T$, we have
	\begin{align}
\left\|Q_{T}-W_T\right\|
&\leq \frac{C}{\zeta_T}\sum_{t\in\Tdm_T}
K\left(\frac{t-{\intp[u_0T]}}{Th}\right)\left\|w_{t-1,T}X_{t-1,T} - w_{t-1}(u_0)X_{t-1}(u_0)\right\|\notag\\
&\leq \frac{1}{\zeta_T}\sum_{t\in\Tdm_T}
K\left(\frac{t-{\intp[u_0T]}}{Th}\right)
\left(
\frac{1}{T} + \left|\frac{t}{T}-u_0\right|
\right)\sum^{p}_{i=1}A_{t-i}.\label{eq:p13}
\end{align}
Notice that  for $t\in\Tdm_T$, the inequality 
\begin{align}
\left|\frac{t}{T}-u_0\right|\leq\left|\frac{t-{\intp[u_0T]}}{T}\right| 
+ \left|\frac{{\intp[u_0T]}}{T}-u_0\right|
\leq C_K h+\frac{1}{T}\notag%\label{eq:p17}
\end{align}
holds, so \eqref{eq:p13} is further bounded by 
\begin{align}
\left\|Q_{T}-W_T\right\|
&\leq \frac{C}{\zeta_T}
\left(C_Kh + \frac{2}{T}\right)
\sum_{t\in\Tdm_T}
K\left(\frac{t-{\intp[u_0T]}}{Th}\right)
\sum^{p}_{i=1}A_{t-i}.\notag
\end{align}
Since $Th\to\infty$, $h$ converges to zero slower than $1/T$. Therefore, {the leading term of  
$\|Q_{T}-W_T\|$
is 
\begin{align}
	\frac{CC_Kh}{\zeta_T}
	\sum_{t\in\Tdm_T}
	K\left(\frac{t-\intp[u_0 T]}{Th}\right)
	\sum^{p}_{i=1}A_{t-i}.\label{eq:p14}
	\end{align}}
	Recall that for sufficient large $T$, we have $\Tdm_T=\{{\intp[u_0T]-\intp[C_KTh],...,\intp[u_0T]+\intp[C_KTh]}\}$
	and  set  \blue{$n=\intp[Th]$} in Lemma \ref{lem:At}. Then, \eqref{eq:p14} is bounded by
	\begin{align}
C \max_{x\in[-C_K,C_K]}K(x)\sum^{p}_{i=1}
\frac{h^{1/2}}{T^{1/2}}
\sum_{{t\in\Tdm_T}}
A_{t-i}
&=_d C \sum^{p}_{i=1}
\frac{h^{1/2}}{T^{1/2}}
\sum_{t=1}^{2C_K n+1}
A_{t-i}\notag\\
&=O_p\left(\frac{h^{1/2}}{T^{1/2}}\blue{a_{\intp[Th]}} + T^{1/2}h^{3/2}\blue{b_{\intp[Th]}}\right).\label{eq:new_evl1}
\end{align}
Under Assumption \ref{ass:n4} and by Lemma \ref{lem:At},
the right hand side of \eqref{eq:new_evl1} converges to zero.
Next, we derive the limit of the {second} term in \eqref{eq:p12}.
Consider the following decomposition:
\begin{align}
&\sum_{t\in\Tdm_T}K\left(\frac{t-{\intp[u_0T]}}{Th}\right)w_{t-1,T}\delta_{t,T}(v)\notag\\
&= \sum_{t\in\Tdm_T}K\left(\frac{t-{\intp[u_0T]}}{Th}\right)\E\left[w_{t-1,T}\delta_{t,T}(v)|\Filt_{t-1,T}\right]\label{delta_decomp_1}\\
&+  \sum_{t\in\Tdm_T}K\left(\frac{t-{\intp[u_0T]}}{Th}\right)
\left\{w_{t-1,T}\delta_{t,T}({v}) - \E\left[w_{t-1,T}\delta_{t,T}(v)|\Filt_{t-1,T}\right]\right\}.\label{delta_decomp_2}
\end{align}

{
The quantity \eqref{delta_decomp_1} is further decomposed as
\begin{align}
	&\sum_{t\in\Tdm_T}K\left(\frac{t-{\intp[u_0T]}}{Th}\right)\E\left[w_{t-1,T}\delta_{t,T}({v})|\Filt_{t-1,T}\right]\notag\\
	&=\sum_{t\in\Tdm_T}K\left(\frac{t-{\intp[u_0T]}}{Th}\right)\E\left[w_{t-1}(u_0)\tilde \delta_{t,T}(u_0, v)|\Filt_{t-1,T}\right]\label{delta_decomp_1_A}\\
	& + \sum_{t\in\Tdm_T}K\left(\frac{t-{\intp[u_0T]}}{Th}\right)
	\left\{\E\left[w_{t-1,T}\delta_{t,T}(v)|\Filt_{t-1,T}\right] - \E\left[w_{t-1}(u_0)\tilde \delta_{t,T}(u_0, v)|\Filt_{t-1,T}\right]\right\}.\label{delta_decomp_1_B}
\end{align}
The quantity given by  \eqref{delta_decomp_1_A} is the leading term. 
The other summand, given by \eqref{delta_decomp_1_B}, represents the difference between the non stationary and stationary processes.
By the proof of Lemma \ref{lem:4}, the term \eqref{delta_decomp_1_A} converges to $(f(0)/2) v^\top \Sigma(u_0)v$ almost surely.
By the same argument as in the proof of Lemma \ref{lem:3}, the term \eqref{delta_decomp_1_B} is bounded by
$\Gamma_T+\Lambda_T + \widetilde\Lambda_T$, where
\begin{align}
	&\Gamma_T := \frac{f(0)}{2Th}\sum_{t\in\Tdm_T}K\left(\frac{t-{\intp[u_0T]}}{Th}\right)
	\left|v^\top\left\{w_{t-1,T}X_{t-1,T}X^\top_{t-1,T} - w_{t-1}(u_0)X_{t-1}(u_0)X^\top_{t-1}(u_0)\right\} v\right|,\notag\\
	&\Lambda_T := \frac{1}{2}\sum_{t\in\Tdm_T}w_{t-1,T}K\left(\frac{t-{\intp[u_0T]}}{Th}\right)\int^{\zeta_T^{-1}\left|v^\top X_{t-1,T}\right|}_{0}s^2 |f'(s^{**})|ds,\notag\\
	&\widetilde\Lambda_T := \frac{1}{2}\sum_{t\in\Tdm_T}w_{t-1}(u_0)K\left(\frac{t-{\intp[u_0T]}}{Th}\right)\int^{\zeta_T^{-1}\left|v^\top X_{t-1}(u_0)\right|}_{0}s^2 |f'(s^*)|ds,\notag
\end{align}
and $s^{*}, s^{**}$ are intermediate  values. It is easily shown that $\Lambda_T$ and $\widetilde\Lambda_T$ converges to zero in probability.
Regarding $\Gamma_T$, a Taylor expansion of the function  $g(x)(v^\top x)^2$ leads
\begin{align}
	&\left|v^\top\left\{w_{t-1,T}X_{t-1,T}X_{t-1,T} - w_{t-1}(u_0)X_{t-1}(u_0)X_{t-1}(u_0)\right\} v\right|
	%&=\left|H(X_{t-1,T}) - H(X_{t-1}(u_0))\right|\notag\\
	%&\leq \sup_{x\in\R^p}\|H'(x)\|\|X_{t-1,T} -X_{t-1}(u_0)\|\notag\\
	%&= \sup_{x\in\R^p}\|g'(x)(v^\top x)^2 + 2g(x)(v^\top x)v\|\|X_{t-1,T} -X_{t-1}(u_0)\|\notag\\
	%&= \|v\|^2\sup_{x\in\R^p}\left\{\|g'(x)\|\|x\|^2 + 2g(x)\|x\|\right\}\|X_{t-1,T} -X_{t-1}(u_0)\|\notag\\
	\leq O(h)\sum^{p}_{i=1}A_{t-i} \quad( t\in\Tdm_T).\notag
\end{align}
Therefore, 
\begin{align}
	\Gamma_T &\leq \frac{f(0)CO(h)}{2Th}\sum_{t\in\Tdm_T}K\left(\frac{t-{\intp[u_0T]}}{Th}\right)
	\sum^{p}_{i=1}A_{t-i}
	= \sum^{p}_{i=1}O_p\left(\frac{1}{T}\sum_{t=p+1}^{2C_K Th+1}
	A_{t-i}\right) = O_p\left(\frac{\blue{a_{\intp[Th]}}}{T}+h\blue{b_{\intp[Th]}}\right).\label{gam_eval}
\end{align}
If $\blue{a_{\intp[Th]}}/T$ diverges to infinity, then $h\blue{a_{\intp[Th]}}$ also diverges, and this contradicts to Assumption \ref{ass:n4}. So we have $\blue{a_{\intp[Th]}}/T\to 0$. On the other hand, $Th\to\infty$ and 
$Th^3\blue{b_{\intp[Th]}}^2=Th(h\blue{b_{\intp[Th]}})^2\to 0$ imply that $h\blue{b_{\intp[Th]}}\to 0$.
%	Since $(3+2\lambda)/(1+2\lambda)<(1+\lambda)/\lambda$ for any $\lambda\in(0,1]$, we have $\limsup_{T\to\infty} T^\lambda h^{1+\lambda}<\infty$.
Hence, the quantity \eqref{gam_eval} converges to zero.
Therefore, the term \eqref{delta_decomp_1_B} converges to zero.
Thus, the quantity \eqref{delta_decomp_1} converges to 
$(f(0)/2) v^\top \Sigma(u_0)v$ almost surely.
}

{
For \eqref{delta_decomp_2},  similar argument as in Lemma \ref{lem:4} shows that 
\begin{align}
	&\E\left(\left|\sum_{t\in\Tdm_T}K\left(\frac{t-{\intp[u_0T]}}{Th}\right)
	\left\{w_{t-1,T}\delta_{t,T}({v}) - \E\left[w_{t-1,T}\delta_{t,T}(v)|\Filt_{t-1,T}\right]\right\}\right|^2\right)\notag\\
	&\leq 2\sum_{t\in\Tdm_T}K\left(\frac{t-{\intp[u_0T]}}{Th}\right)
	\E\left(w_{t-1,T}^2\delta_{t,T}({v})^2\right)\notag\\
	&\leq \frac{C}{(Th)^{3/2}}\sum_{t\in\Tdm_T}K\left(\frac{t-{\intp[u_0T]}}{Th}\right)
	\|{v}\|^3\to0.\notag
\end{align}
This implies that the term \eqref{delta_decomp_2} converges to zero in probability.
Thus, we conclude that the last term in \eqref{eq:p12} converges to $(f(0)/2) v^\top \Sigma(u_0)v$ in probability.
}

\medskip

{By collecting all the previous results,  we obtain the approximation of \eqref{eq:p12} as follows:
\begin{align}
	L_{T}(u_0, v)
	= v^\top W_T
	% + \gamma^\top \tilde Q_2 + \tilde Q_3 
	+f(0)v^\top \Sigma(u_0)v
	+o_p(1)\label{eq:p21}
\end{align}
at each ${v}\in\R^p$. 

{
By an application of Lemma~\ref{Lemma:Convexity}, the point-wise convergence established in (\ref{eq:p21}) becomes
uniform convergence on compact subsets of $\mathbb{R}^p$. Specifically, letting
	\begin{align*}
	\lambda_T(u_0, v):=L_{T}(u_0, v)-v^\top W_T,
\end{align*}
and noting that, for every fixed $v\in\mathbb{R}^p$,  $\lambda_T(u_0, v)\to\lambda(u_0, v):=f(0) v^\top  \Sigma(u_{0}) v$, in probability, the desired uniform convergence follows immediately from Lemma~\ref{Lemma:Convexity}, since the functions $\lambda_T(u_0,\cdot)$ are convex.}

%	Equivalently 
%	\begin{align*}
%		L_{T}(u_0, v)&=v^\top W_T+f(0) v^\top  \Sigma(u_{0}) v+r_T(u_0, v)\\
%		&=f(0) \{v-\eta_T(u_{0})\}^\top  \Sigma(u_{0}) \{v-\eta_T(u_{0})\}-f(0) \eta_T(u_{0})^\top  \Sigma(u_{0})\eta_T(u_{0})+r_T(u_0, v),
%	\end{align*}
%	where $\eta_T(u_{0}):=-\frac{1}{2f(0)}\Sigma^{-1}(u_{0})W_T$ and for each compact set $K$ in $\R^p$
%	\begin{align*}
%		\sup_{v\in K}\left|r_T(u_0, v)\right|\to 0,
%	\end{align*}
%	in probability. Since $\eta_T(u_{0})$ converges in distribution and therefore, is stochastically bounded, we can choose the closed ball $B_T(u_{0})$ with center $\eta_T(u_{0})$ and radius $\delta$ such that
%	\begin{align*}
%		\Delta_T:=\sup_{v\in B_T(u_{0})}\left|r_T(u_0, v)\right|\to 0, 
%	\end{align*}
%	in probability.
%	On the other hand, if $v$  lies outside the ball $B_T(u_{0})$, the exact same argument as the latter part of the proof of Theorem 1 of \citet{Pollard(1991)} leads to the conclusion that  for large $T$, we can choose $\delta>0$ for  which the minimum of $L_{T}(u_0, v)$ cannot be attained at any $v$ within $\left|v-\eta_T(u_{0})\right|>\delta$, with probability one.
%

Recalling  Lemma \ref{lem:3}, ${L_{T}(u_0; v)}$ converges to
\[
v^\top W + f(0)v^\top \Sigma(u_0)v
\]
in distribution as $T\to\infty$, %\label{eq:lim_G}
where $W\sim N(0_p, \kappa_2 \Omega(u_0))$.
The minimum of the previous display  is attained at 
\[
v= -\frac{1}{4f(0)}\Sigma^{-1}(u_0)W\sim N\left(0_p, \frac{\kappa_2}{4f^2(0)}\Sigma^{-1}(u_0)\Omega(u_0)\Sigma(u_0)^{-1}\right)
\]
almost surely, and hence
\[
\sqrt{Th}[{\hat\beta_{T}(u_0)-\beta(u_0)}]\dlim  N\left(0_p,\frac{\kappa_2}{4f^2(0)}
\Sigma^{-1}(u_0)\Omega(u_0)\Sigma^{-1}(u_0)\right)\quad(T\to\infty).
\]}

\end{proof}

\subsection{Proof of Theorem \ref{thm:BS}}
\begin{proof}
By the proof of Theorem \ref{thm:1}, we obtain the expansion of $\hat\beta_T(u_0)$ as 
\begin{align}
\zeta_T(\hat\beta_T(u_0) - \beta(u_0)) = -\left[2f(0)\Sigma(u_0)\right]^{-1}W_T+o_p(1).\label{eq:exp_est}
\end{align}
On the other hand, consider the function 
\begin{align}
L_T^*(u_0; v) := \sum_{t\in \Tdm_T}z_tK\left(\frac{t-\lfloor u_0 T\rfloor}{Th}\right)w_{t-1, T}\left(
\left|\epsilon_t - \zeta_T^{-1} v^\top X_{t-1,T}\right|
-|\epsilon_t|
\right).\notag
\end{align}
Then, it is shown that $\hat v^* := \zeta_T(\hat\beta_T^*(u_0) - \beta(u_0))$ is a minimizer of 
$L_T^*(u_0; v)$, by the similar reason for 
$L_{T}(u_0; v)$ and $\zeta_T(\hat\beta_T(u_0) - \beta(u_0))$.
We expand $L_T^*(u_0; v)$ as 
\begin{align}
L_T^*(u_0; v)
= {v^\top Q_{T}^* +2\sum_{t\in\Tdm_T}z_tw_{t-1,T}K\left(\frac{t-{\intp[u_0T]}}{Th}\right)\delta_{t,T}(v)}, \label{eq:26}
\end{align}
where 
\begin{align}
Q_{T}^* = -\frac{1}{\zeta_T}\sum_{t\in\Tdm_T}z_tw_{t-1,T}K\left(\frac{t-{\intp[u_0T]}}{Th}\right)\mathrm{sign}(\epsilon_t)X_{t-1,T}\notag
\end{align}
and $\delta_{t,T}$ is defined in \eqref{def:def_delta}.
Let 
\[
W_T^* := -\frac{1}{\sqrt{Th}}\sum_{{t\in\Tdm_T}}z_t w_{t-1}(u_0)K\left(\frac{t-{\intp[u_0T]}}{Th}\right)\mathrm{sign}(\epsilon_t)X_{t-1}(u_0).
\]
By the similar way as \eqref{eq:p13}, the difference between $W_T^*$ and $Q_T^*$ is evaluated as 
\begin{align}
\left\|Q_{T}^*-W_T^*\right\|
\leq \frac{C}{\zeta_T}\sum_{t\in\Tdm_T}
K\left(\frac{t-{\intp[u_0T]}}{Th}\right)
\left(
\frac{1}{T} + \left|\frac{t}{T}-u_0\right|
\right) z_t\sum^{p}_{i=1}A_{t-i}, \label{eq:a27}
\end{align}
and the leading term of \eqref{eq:a27} is 
\begin{align}
\frac{Ch}{\zeta_T}
\sum_{t\in\Tdm_T}
K\left(\frac{t-\intp[u_0 T]}{Th}\right)
\sum^{p}_{i=1}z_t A_{t-i}.\label{eq:a28}
\end{align}
Note that 
\begin{align}
\frac{h^{1/2}}{T^{1/2}}\sum_{t=1}^{2C_K Th+1}
z_t A_{t-i}
\leq 
\frac{h^{1/2}}{T^{1/2}}
\left\{
\sum_{t=1}^{2C_K Th+1}z_t^2
\right\}^{1/2}
\left\{
\sum_{t=1}^{2C_K Th+1}A_{t-i}^2
\right\}^{1/2}\notag\\
=
\left\{
\frac{1}{Th}\sum_{t=1}^{2C_K Th+1}z_t^2
\right\}^{1/2}
\left\{
h^2
\sum_{t=1}^{2C_K Th+1}A_{t-i}^2
\right\}^{1/2}.\label{eq:a29}
\end{align}
The first factor in the right hand side of \eqref{eq:a29} is $O_p(1)$ by the law of large numbers,and the second factor is $o_p(1)$ by Lemma \ref{lem:At} and the assumption in Theorem \ref{thm:BS}.
Thus, \eqref{eq:a28} converges to zero in probability by the same reason as \eqref{eq:new_evl1}, 
and we obtain $Q_T^* = W_T^*+o_p(1)$.
On the other hand, by noting the independence of $z_t$ and by the similar argument as in the proof of Theorem \ref{thm:1}, the second term of \eqref{eq:26} is shown to converges to 
$(f(0)/2) v^\top \Sigma(u_0)v$ a.s.
These facts leads the expression 
\begin{align}
\zeta_T(\hat\beta_T^*(u_0) - \beta(u_0)) = -\left[2f(0)\Sigma(u_0)\right]^{-1}W_T^*+o_p(1).\label{eq:exp_est_BS}
\end{align}
Then, \eqref{eq:exp_est} and \eqref{eq:exp_est_BS} yields
\begin{align}
&\zeta_T(\hat\beta_T^*(u_0) - \hat\beta_T(u_0))\notag\\
&=\left[2f(0)\Sigma(u_0)\right]^{-1}
\frac{1}{\sqrt{Th}}\sum_{{t\in\Tdm_T}}(z_t-1) w_{t-1}(u_0)K\left(\frac{t-{\intp[u_0T]}}{Th}\right)\mathrm{sign}(\epsilon_t)X_{t-1}(u_0){+o_p(1)}.\label{eq:a33}
\end{align}
By recalling the unit mean and variance of $z_t$, the right hand side of \eqref{eq:a33} is shown to have the same limit distribution as in \eqref{eq:lim_thm1}, conditional on $Y_1(u_0),...,Y_T(u_0)$.
This complete the proof. 
\end{proof}

\subsection{Proof of \eqref{BS_exp_1:eq1}}
\label{prrof of joint}

Note that when the sample size $T$ diverges to infinity ($h\to 0$), $\hat\beta_T(u_1)$ and $\hat\beta_T(u_1)$ ($u_1>u_2$) are asymptotically uncorrelated, but when the sample size is moderate compared to the difference $u_1-u_2$, they are correlated. In fact, using the expansion of $\hat\beta_T(u))$ in (\ref{eq:exp_est}), we can calculate
	\begin{align*}
		&\zeta_T^2{\rm Cov}\left(\hat\beta_T(u_1) - \beta(u_1),\hat\beta_T(u_2) - \beta(u_2)\right)\\
		&=\left\{4 f(0)^2\right\}^{-1}\Sigma^{-1}(u_1){\rm Cov}\left(W_T(u_1),W_T(u_2)\right)\Sigma^{-1}(u_2)+o(1)\\
		&=\left\{4 f(0)^2\right\}^{-1}\Sigma^{-1}(u_1)\mathbb{E}\left\{W_T(u_1)W_T(u_2)\right\}
		\Sigma^{-1}(u_2)+o(1),
	\end{align*}
	where $W_{T}(u)$ has been defined in Lemma \ref{lem:3} . But 
	\begin{align*}
		&\mathbb{E}\left\{W_T(u_1)W_T(u_2)\right\}\\
		&=\mathbb{E}\left\{w_{t-1}(u_1)w_{t-1}(u_2)X_{t-1}(u_1)X_{t-1}^{\top}(u_2)\right\}\frac{1}{Th}\sum_{{t\in\Tdm_T}}K\left(\frac{t-{\intp[u_1T]}}{Th}\right)K\left(\frac{t-{\intp[u_2T]}}{Th}\right)\\
		&\approx \mathbb{E}\left\{w_{t-1}(u_1)w_{t-1}(u_2)X_{t-1}(u_1)X^{\top}_{t-1}(u_2)\right\}\int^{C_K}_{-C_K}K\left(v\right)K\left(v-\frac{u_1-u_2}{h}\right)dv,
	\end{align*}
	we see that $\hat\beta_T(u_1)$ and $\hat\beta_T(u_1)$ ($u_1>u_2$) are asymptotically uncorrelated, but correlated for the moderate sample size.

\subsection{Additional simulation results}

\subsubsection{Estimation} 

We also carried out the same simulation experiments for the data from tvAR(2) model 
\begin{align}
Y_{t,T}= 0.8 \sin(2 \pi t/T) Y_{t-1, T} + 0.2 \sin(2 \pi (t/T+0.1)) Y_{t-2, T} + \epsilon_{t}.\label{eq:AR2}
\end{align}
%In this case,  in addition to the mean absolute error \eqref{mae}, 
%we also calculate mean squared error defined as
%\begin{align}
%\mathrm{MSE} = \frac{1}{T} \sum_{i=1}^{T} \| \hat{\beta}(\frac{i}{T})-\beta(\frac{i}{T}) \|_{2}.\notag
%\end{align}
%The other settings are the same as in the tvAR(1) case.
The estimated mean squared errors are calculated as before--recall  \eqref{eq:est_mae}.
Table  \ref{tbl_MSE2} show the results in the tvAR(2) case.

%\begin{table}[htbp]
%\centering
%\caption{Estimated mean absolute errors for the model \eqref{eq:AR2} based on 1000 simulations and using different sample sizes. Minimum in each line is indicated by boldface fonts.}\label{tbl_MAE2}
%\medskip
%\begin{tabular}{lccccccc}
%$\epsilon_t\sim N(0,1)$ & L2 & LAD & LSW1c1 & LSW1c2 & LSW2q1 & LSW2q2 & LSW3\\\hline
%$T=100$ 		& \textbf{0.4080} & 0.4362 & 0.4659 & 0.4241 & 0.4558 & 0.4793 & 0.4398 \\
%$T=500$ 		& \textbf{0.1930} & 0.2214 & 0.3219 & 0.2399 & 0.2632 & 0.3091 & 0.2386 \\
%$T=1000$ 	& \textbf{0.1404} & 0.1681 & 0.2900 & 0.1974 & 0.2152 & 0.2684 & 0.1886 \\\hline
%$\epsilon_t\sim t_2$&  &  &  &  &  &  &  \\\hline
%$T=100$ 		& 0.4131 & 0.3905 & 0.4464 & \textbf{0.3796} & 0.3959 & 0.4046 & 0.4224 \\
%$T=500$ 		& 0.1990 & \textbf{0.1754} & 0.3059 & 0.2139 & 0.1844 & 0.2046 & 0.2039 \\
%$T=1000$ 	& 0.1461 & \textbf{0.1229} & 0.2721 & 0.1739 & 0.1343 & 0.1566 & 0.1575 \\\hline
%$\epsilon_t\sim $Cauchy &  &  &  &  &  &  &  \\\hline
%$T=100$ 		& 0.4353 & 0.3810 & 0.4700 & 0.3853 & 0.3928 & \textbf{0.3794} & 9.2486 \\
%$T=500$ 		& 0.2204 & 0.1808 & 0.3118 & 0.2174 & 0.1690 & \textbf{0.1589} & 0.2067 \\
%$T=1000$ 	& 0.1616 & 0.1272 & 0.2715 & 0.1748 & 0.1119 & \textbf{0.1063} & 0.1475 \\\hline
%\end{tabular}
%\end{table}
\begin{table}[H]
\centering
\caption{MSE for the model \eqref{eq:AR2} based on  1000 simulations and using different sample sizes. Minimum in each line is indicated by boldface fonts. The * indicates a large value.} \label{tbl_MSE2}
\medskip
\begin{tabular}{lcccccc}\hline
	$\epsilon_t\sim N(0,1)$ & L2 & LAD & LSW1c1 & LSW1c2 & LSW2q1 & LSW2q2 \\\hline
	$T=100$ &\textbf{ 0.1293} & 0.1472 & 0.1710 & 0.1399 & 0.1508 & 0.1659 \\
	$T=500$ & \textbf{0.0302} & 0.0390 & 0.0519 & 0.0397 & 0.0419 & 0.0478 \\
	$T=1000$ & \textbf{0.0162} & 0.0227 & 0.0326 & 0.0240 & 0.0254 & 0.0297 \\\hline
	$\epsilon_t\sim t_2$ & L2 & LAD & LSW1c1 & LSW1c2 & LSW2q1 & LSW2q2 \\\hline
	$T=100$ & 0.1363 & \textbf{0.1185 }& 22.5478 & 0.1309 & 0.1268 & 0.1433 \\
	$T=500$ & 0.0343 & \textbf{0.0266} & 0.0582 & 0.0326 & 0.0278 & 0.0328 \\
	$T=1000$ & 0.0186 &\textbf{ 0.0131} & 0.0347 & 0.0182 & 0.0148 & 0.0178 \\\hline
	$\epsilon_t\sim t_{1.5}$ & L2 & LAD & LSW1c1 & LSW1c2 & LSW2q1 & LSW2q2 \\\hline
		$T=100$     & 0.1469 & \textbf{0.1183} & 0.5895 & 0.1481 & 0.1231 & 0.1379\\
		$T=500$     & 0.0375 & 0.0268 & 0.0668 & 0.0334 &\textbf{ 0.0245 }& 0.0274\\
		$T=1000$    & 0.0207 & 0.0131 & 0.0402 & 0.0188 &\textbf{ 0.0119} & 0.0143\\\hline
	$\epsilon_t\sim \text{Cauchy}$ & L2 & LAD & LSW1c1 & LSW1c2 & LSW2q1 & LSW2q2 \\\hline
	$T=100$ & 0.1725 & \textbf{0.1347} & 180.7439 & 133.0343 & 0.1578 & 0.1928 \\
	$T=500$ & 0.0554 & 0.0353 & 0.1184 & 0.0408 & 0.0247 & \textbf{0.0230 }\\
	$T=1000$ & 0.0274 & 0.0178 & 0.0741 & 0.0228 & 0.0111 & \textbf{0.0104} \\\hline
%	$\epsilon_t\sim \text{S$\alpha$S}$ & L2 & LAD & LSW1c1 & LSW1c2 & LSW2q1 & LSW2q2 \\\hline
%	$T=100$ & (*1) & 0.2119 & (*2) & (*3) & 3660.1435 & 16373.9868 \\
%	$T=500$ & 2.8085 & \textbf{0.0731} & * & * & 1.4736 & 44.6021 \\
%	$T=1000$ & 19.4045 &\textbf{ 0.0448} & * & *  & 0.3457 & 0.3772 \\\hline
\end{tabular}
%\begin{tabular}{lccccccc}
%$\epsilon_t\sim N(0,1)$ & L2 & LAD & LSW1c1 & LSW1c2 & LSW2q1 & LSW2q2 & LSW3\\\hline
%$T=100$ 		& \textbf{0.3172} & 0.3391 & 0.3712 & 0.3318 & 0.3586 & 0.3804 & 0.3424 \\
%$T=500$ 		& \textbf{0.1507} & 0.1722 & 0.2589 & 0.1880 & 0.2083 & 0.2490 & 0.1845 \\
%$T=1000$ 	& \textbf{0.1095} & 0.1303 & 0.2339 & 0.1550 & 0.1706 & 0.2169 & 0.1451 \\\hline
%$\epsilon_t\sim t_2$&  &  &  &  &  &  &  \\\hline
%$T=100$ 		& 0.3225 & 0.3043 & 0.3590 & \textbf{0.2997} & 0.3094 & 0.3187 & 0.3295 \\
%$T=500$ 		& 0.1561 & \textbf{0.1382} & 0.2481 & 0.1702 & 0.1450 & 0.1623 & 0.1579 \\
%$T=1000$ 	& 0.1143 & \textbf{0.0967} & 0.2213 & 0.1391 & 0.1053 & 0.1244 & 0.1212 \\\hline
%$\epsilon_t\sim $Cauchy &  &  &  &  &  &  &  \\\hline
%$T=100$ 		& 0.3465 & 0.3024 & 0.3825 & 0.3093 & 0.3064 & \textbf{0.2970} & 6.5867 \\
%$T=500$ 		& 0.1766 & 0.1462 & 0.2545 & 0.1755 & 0.1338 & \textbf{0.1262} & 0.1612 \\
%$T=1000$ 	& 0.1299 & 0.1035 & 0.2216 & 0.1415 & 0.0887 & \textbf{0.0844} & 0.1142 \\\hline
%\end{tabular}
\end{table}

{
Overall, we observe that the estimators LSW2q1 and LSW2q2  perform the best in the case of $t_{1.5}$ and Cauchy distributions and for large sample sizes.  In general, the estimators LSW1c1 and LSW2c2 attain large MSE. }
%For the case of the $t$ and symmetric $\alpha$-stable distributions  the unweighted LAD estimator performs the best for large sample sizes. However, LAD estimators do not have the property of  asymptotic normality, and hence, they  are not robust to outliers.}

\subsubsection{Coverage probability of confidence region}
We next consider the model tvAR(2) model \eqref{eq:AR2}, and construct the confidence region for $\beta(u_0) = (\beta_{1}(u_0), \beta_{2}(u_0))^\top$. 
Throughout this subsection, we fix $u_0 = 0.5$.
For this purpose, define a criterion function
$\hat{C}_{T,B}(\beta) := Th(\beta-\hat{\beta}_T(u_0))^\top \hat{\Xi^*}^{-1}_{T,B}(u_0)^{-1}(\beta-\hat{\beta}_T(u_0))$ $(\beta\in\R^2)$,
where $\hat{\beta}_T(u_0)$ is LSWLADE based on an observed stretch $Y_{1,T},...,Y_{T,T}$, 
and $\hat{\Xi^*}_{T,B}(u_0)$ is the bootstrap asymptotic covariance matrix estimator of the asymptotic covariance matrix of $\sqrt{Th}(\hat{\beta}_T(u_0) - \beta(u_0))$ obtained by Theorem \ref{thm:BS}. 
Then, the $100(1-\delta)$\%-confidence region ($\delta\in(0,1)$) of $\beta(u_0)$ is calculated as
$\hat{R}_{T,B} := \left\{
b\in \R^2: \hat{C}_{T,B}(b)\leq q_{2,\delta}
\right\}$,
where $q_{2,\delta}$ is the upper $\delta$ quantile of $\chi^2_2$-distribution.
Table \ref{tbl:a2} shows the empirical coverage probability of confidence region $\hat{R}_{T,B}$ based on 1000 replications 
(i.e., $\sum^{1000}_{l=1}\mathbb{I}(\beta(u_0)\in\hat{R}_{T,B}^{(l)})/1000$, where
$\hat{R}_{T,B}^{(l)}$ is the confidence region in $l$th iteration)
with theoretical coverage probabilities 90\% and 95\%.
Generally, the performance of the proposed confidence region becomes better as sample size increases.
In particular, the empirical coverage probability is quite close to the true one when the error follows a heavy-tailed distribution. 

%Cauchy distribution  and $T=1000$.
%Then, we can see the robustness of our method even when the error distribution departs far from normality.

\begin{table}[H]
\centering
\caption{Empirical coverage probabilities of $\hat{R}_{T,M}$ for $\beta(u_0)$}\label{tbl:a2}
\medskip
%\begin{tabular}{ccc}
%\begin{tabular}{lcc}
%	\multicolumn{3}{l}{$\epsilon_t\sim N(0,1)$}\\
%	$T$ & 90\% & 95\%\\\hline
%	100 & 0.879 & 0.927 \\
%	500 & 0.878 & 0.933 \\
%	1000 & 0.889 & 0.936 \\\hline
%\end{tabular}
%&
%\begin{tabular}{lcc}
%	\multicolumn{3}{l}{$\epsilon_t\sim t_2$}\\
%	$T$ & 90\% & 95\% \\\hline
%	100 & 0.887 & 0.933 \\
%	500 & 0.902 & 0.940 \\
%	1000 & 0.915 & 0.950 \\\hline
%\end{tabular}
%&
%\begin{tabular}{lcc}
%	\multicolumn{3}{l}{$\epsilon_t\sim$ Cauchy}\\
%	$T$ & 90\% & 95\% \\\hline
%	100 & 0.860 & 0.920 \\
%	500 & 0.884 & 0.945 \\
%	1000 & 0.902 & 0.947 \\\hline
%\end{tabular}
%\end{tabular}
 \begin{tabular}{ccc}
	\multicolumn{3}{l}{$\epsilon_t\sim N(0,1)$}\\\hline
	$T$ & 90\% & 95\% \\\hline
%	100 & 0.868 & 0.925 \\
	500 & 0.873 & 0.933 \\
	1000 & 0.887 & 0.935 \\\hline
\end{tabular} 
\begin{tabular}{ccc}
	\multicolumn{3}{l}{$\epsilon_t\sim t_2$}\\\hline
	$T$ & 90\% & 95\% \\\hline
%	100 & 0.883 & 0.930 \\
	500 & 0.911 & 0.938 \\
	1000 & 0.915 & 0.955 \\\hline
\end{tabular} 
\\
\begin{tabular}{ccc}
	\multicolumn{3}{l}{$\epsilon_t\sim t_{1.5}$}\\\hline
	$T$ & 90\% & 95\% \\\hline
%	100     & 0.895	& 0.933\\
	500     & 0.909	& 0.948\\
	1000	& 0.909	& 0.942\\\hline
\end{tabular}
\begin{tabular}{ccc}
	\multicolumn{3}{l}{$\epsilon_t\sim \text{Cauchy}$}\\\hline
	$T$ & 90\% & 95\% \\\hline
%	100 & 0.860 & 0.915 \\
	500 & 0.890 & 0.942 \\
	1000 & 0.907 & 0.944 \\\hline
\end{tabular} 
%\begin{tabular}{ccc}
%	\multicolumn{3}{l}{$\epsilon_t\sim \text{S$\alpha$S}$}\\\hline
%	$T$ & 90\% & 95\% \\\hline
%%	100 & 0.833 & 0.864 \\
%	500 & 0.868 & 0.921 \\
%	1000 & 0.892 & 0.940 \\\hline
%\end{tabular}
\end{table}

%\bibliographystyle{chicagoa}
%\bibliography{RobustBIB.bib}

\end{document}